\documentclass[12pt,oneside]{amsart}
\usepackage[T1]{fontenc}
\usepackage[utf8]{inputenc}
\usepackage{amsmath,amssymb,amsthm,mathtools}
\usepackage{lmodern}
\usepackage[all,cmtip]{xy}
\usepackage{graphicx}

\usepackage[a4paper,left=20mm,right=20mm,top=25mm,bottom=25mm]{geometry}
\usepackage{enumitem}
\usepackage{microtype}
\usepackage[hidelinks]{hyperref}
\usepackage{bookmark}
\setlist{leftmargin=2em,itemsep=0.25em,topsep=0.4em}
\allowdisplaybreaks[2]
\numberwithin{equation}{section}
\newtheorem{theorem}{Theorem}[section]
\newtheorem{lemma}[theorem]{Lemma}

\theoremstyle{definition}
\newtheorem{definition}[theorem]{Definition}
\newtheorem{example}[theorem]{Example}
\theoremstyle{remark}

\newcommand{\C}{\mathbb C}
\newcommand{\D}{\mathbb D}
\newcommand{\E}{\mathbb E}
\newcommand{\Set}{\mathbf{Set}}

\newcommand{\id}{\operatorname{id}}

\newcommand{\im}{\operatorname{im}}

\newcommand{\op}{\mathrm{op}}

\newcommand{\colvec}[2]{\begin{pmatrix}#1\\#2\end{pmatrix}}

\newcommand{\A}{\mathbb A}

\newcommand{\Bic}{\operatorname{Bic}}

\newcommand{\DiagLin}{\mathbf{DiagLin}}
\newcommand{\BiCart}{\mathbf{BiCart}_0}

\newcommand{\N}{\mathbb N}

\newcommand{\End}{\operatorname{End}}

\newcommand{\Mat}{\operatorname{Mat}}
\newcommand{\FinGrp}{\mathbf{FinGrp}}
\newcommand{\FinPInj}{\mathbf{FinPInj}}

\hypersetup{pdftitle={Partial linearity: diagolinear categories},pdfauthor={Roy Ferguson and Zurab Janelidze}}
\title[Partial linearity: diagolinear categories]{Partial linearity: diagolinear categories}
\author{Roy Ferguson}
\address{Department of Computer Science, University of Cape Town,
Private Bag X3, Rondebosch 7701, South Africa}
\email{royf326@gmail.com}
\author{Zurab Janelidze}
\address{Mathematics Division, Department of Mathematical Sciences,
Stellenbosch University, Private Bag X1, Matieland 7602, South Africa}
\address{National Institute for Theoretical and Computational Sciences
(NITheCS), Private Bag X1, Matieland 7602, South Africa}
\email{zurab@sun.ac.za}
\date{}
\begin{document}
\begin{abstract}
In this paper we introduce and study \emph{diagolinear categories}. These are pointed categories equipped with sum-like and product-like monoidal structures, where diagonal matrices of morphisms between two sequences of objects are (uniquely) realizable as morphisms from the sum of the first sequence to the product of the second sequence. We formulate and prove the coherence theorem for diagolinear categories and study the universal construction that realizes also the non-diagonal matrices (equivalently, turns the two monoidal structures into actual categorical sum and product). We also compare diagolinear categories with duoidal categories. 
\end{abstract}
\maketitle

\section*{Introduction}
\label{sec:introduction}

In a category having finite sums (coproducts) and products, every morphism from a sum of $n$ objects to a product of $m$ objects is uniquely determined by an $m\times n$ matrix of morphisms between individual objects in the two sequences. Moreover, any matrix of such morphisms is represented by a morphism from the $n$-fold sum to the $m$-fold product. The unique morphism from the initial object (the empty sum) to the terminal object (the empty product) is represented by the empty matrix. If this morphism is invertible, the category becomes pointed, which allows consideration of identity matrices. When the morphisms represented by identity matrices are also invertible, we get linear categories in the sense of \cite{Lawvere1992} (see also \cite{MacLane1950,Mitchell1965,FreyScedrov1990,Lawvere_Schanuel_2009}). In this context, the classical formula for multiplication of matrices describes composition of morphisms that matrices represent; in fact, the category becomes enriched in the category of commutative monoids, which provides the addition for the formula of multiplication of matrices. In additive categories \cite{Grothendieck1957} the addition furthermore has additive inverses. Examples of additive categories include categories of $R$-modules, Hilbert spaces and Banach spaces with bounded linear maps, and Fr\'echet spaces and locally convex topological vector spaces with continuous linear maps.

The conceptual ontology behind additive categories briefly described above exhibits categorical contexts where some fragments of the calculus of matrices are present, while others may be missing. Let us refer to the subject of study of such categorical contexts as \emph{partial linearity}. This subject comes up in various places in the literature and relates to applications such as flow-diagram semantics \cite{ArbibManes1980}, the geometry of interaction \cite{Haghverdi2000,HaghverdiScott2006}, the categorical treatment of partial maps \cite{CockettLack2007}, reversible computation \cite{KaarsgaardAxelsenGluck2017}, categorical quantum foundations \cite{Cho2015,Tull2018}, differentiation and Taylor expansion in models of linear logic \cite{Ehrhard2023,EhrhardWalch2025,TassonWalch2026}, and probabilistic computation \cite{BonchiCioffo2026}. In the present paper we pioneer the study of categories where diagonal matrices are realisable (uniquely), but are not necessarily invertible and at the same time not all matrices necessarily have realisations. We replace the usual coproducts and products with sum-like and product-like monoidal structures, and in particular with monoidal sum structures from \cite{Janelidze_2009} and their duals. This framework provides us with the right context to be able to speak of realisability of matrices of morphisms, with a realisation, when it exists, being always unique. We then introduce \emph{diagolinear categories} as those pointed categories equipped with such a pair of structures where diagonal matrices are realisable. Duoidal categories in the sense of Ross Street \cite{Street2012,BookerStreet2013} (see also \cite{Garner2009,AguiarMahajan2010,BataninMarkl2012}) come close to our diagolinear categories, although there are principal differences. Firstly, even in the case when the two monoidal structures in the category are our sum and product structures, the duoidal case does not guarantee pointedness; it can be, however, formulated as a property of realisation of certain matrices, as we show in this paper. In the pointed case, still with the same monoidal sum and monoidal product structures, duoidal categories are diagolinear but the converse is not true. We do show, however, that in the pointed context our property can be regarded as a natural weakening of the duoidal requirement. 

Pointed categories with co-cartesian monoidal sum structure (the usual coproduct) and cartesian monoidal product structure (the usual product) are diagolinear categories. We call these diagolinear categories \emph{bicartesian}. They are precisely those diagolinear categories where \emph{all} matrices are realisable. Examples of a different kind are given by weakly unital categories~\cite{MartinsFerreira2008Thesis}.
There both specified monoidal structures are the same monoidal structure, given by the usual product. If in the first case we have realisability of all matrices, in the second case the identity matrices are invertible, while not all matrices are necessarily realisable. Combining these two scenarios recovers precisely the linear categories. Considering general diagolinear categories seems to be an area of partial linearity that has not been studied before: in all direct instances of diagolinearity that we have identified in the literature, either the specified structures are the categorical coproduct and product, or the identity comparisons are invertible.

In this paper we develop a basic theory of diagolinear categories, including a coherence theorem and reflection to bicartesian diagolinear categories (i.e., pointed categories having finite products and sums). We also exhibit examples illustrating a distinction between diagolinear and duoidal categories, lack of preservation of finite products by the bicartesian reflection (we provde that finite products are preserved), and lack of faithfulness of the components of the unit of the bicartesian reflection.

\tableofcontents

\section{Diagolinear categories}
\label{sec:theory}

Throughout, categories are locally small. We write $gf$ for $g\circ f$.
No symmetry of a monoidal structure is assumed unless explicitly stated.

\subsection{Sum-product categories}

A monoidal sum structure \cite{Janelidze_2009} $\oplus$ on a category $\C$ is defined to be a monoidal structure $(\oplus, 0, \alpha_{\oplus}, \lambda_{\oplus}, \rho_{\oplus})$ where the monoidal unit is an initial object $0$ and the standard binary inclusions 
\[
    \xymatrix{
        X \oplus 0 \ar[r]^{1_X \oplus !} & X \oplus Y & 0 \oplus Y \ar[l]_{! \oplus 1_Y} \\
        X \ar[u]^{\rho_X^{-1}} \ar[ru]_{\iota_1} & & Y \ar[u]_{\lambda_Y^{-1}} \ar[lu]^{\iota_2}
    }
\]
are jointly epimorphic. Dually, a monoidal product structure
$(\otimes,1,\alpha_\otimes,\lambda_\otimes,\rho_\otimes)$ has a
terminal unit $1$, and its canonical projections
\[
 \pi_1=\rho_{\otimes,X}(1_X\otimes!):X\otimes Y\to X,
 \qquad
 \pi_2=\lambda_{\otimes,Y}(!\otimes1_Y):X\otimes Y\to Y
\]
are jointly monomorphic. We will drop the subscript $\oplus$ when
there is no confusion. Associated to each sum structure $\oplus$ is a cover relation $\sqsubset_\oplus$ (a binary relation on the class of morphisms of $\C$) that essentially measures how far the sum structure is from a coproduct structure. Here $f \sqsubset_\oplus g$ if they have a common codomain, and there exists some (necessarily unique) dotted morphism making the following diagram commute.
\[
    \xymatrix{
        X_1 \ar[r]^-{\iota_1} \ar[dr]_f & X_1 \oplus X_2 \ar@{.>}[d] & X_2 \ar[l]_-{\iota_2} \ar[dl]^g \\
        & Y & 
    }
\]

\begin{lemma}[\cite{Janelidze_2009}]\label{lem:symmetric-cover}
A sum structure admits a symmetry if and only if its cover relation
$\sqsubset_\oplus$ is symmetric. When a symmetry exists, it is uniquely
determined by interchanging the canonical inclusions.
\end{lemma}
\begin{proof}
A symmetry interchanges the inclusions, so precomposition with it
interchanges cooperating pairs. Conversely, if the relation is symmetric,
the inclusions into $Y\oplus X$ admit a cooperator in the reversed order.
This gives $\sigma_{X,Y}:X\oplus Y\to Y\oplus X$ interchanging the
inclusions. Its reverse is its inverse, by joint epimorphicity.
Naturality, the unit condition and the symmetry and hexagon identities
follow by testing the inclusions in each iterated sum. These inclusions
are jointly epimorphic by induction, so these tests prove the required
identities.
\end{proof}

For a finite ordered list $X_1,\ldots,X_n$, any parenthesised iteration
of $\oplus$, with any inserted copies of its unit, represents its
$n$-fold sum. The ordinary coherence theorem for monoidal categories
identifies these representatives canonically~\cite[Chapter~VII]{maclane_1998}.
Their inclusions are defined recursively using the binary inclusions
and unitors; they are natural, jointly epimorphic, and preserved by the
canonical identifications. The dual assertions hold for iterated
products. At length zero the sum and product are respectively $0$ and
$1$; at length one their standard representatives are the object itself.

Write $S_n:\C^n\to\C$ for the right-associated $n$-fold sum,
with $S_0=0$ and $S_1=\mathrm{Id}_{\C}$. For $n\geq1$ and
$1\leq s\leq n$, put
\[
 E_s(X)=(0,\ldots,0,X,0,\ldots,0),
\]
where $X$ occupies position $s$. The usual monoidal unit
isomorphisms give a natural isomorphism
\[
 \kappa_s:S_nE_s\xrightarrow{\cong}\mathrm{Id}_{\C}.
\]
Ordinary monoidal coherence makes this isomorphism independent of
the sequence of associators and unitors used to obtain
it~\cite[Chapter~VII]{maclane_1998}.

\begin{theorem}[Determination by one-variable restrictions]
\label{newcoh:restriction}
Let $\C$ have a monoidal sum structure, and let
$F:\C^n\to\C$ be any functor, with $n\geq1$.
Natural transformations $S_n\Rightarrow F$ are determined by their
restrictions along the functors $E_s$. In particular, if
$\theta,\psi:S_n\Rightarrow F$ satisfy
\[
 \theta_{E_s(X)}=\psi_{E_s(X)}
 \qquad\text{for every }s\text{ and }X,
\]
then $\theta=\psi$. Dually, if $\C$ has a monoidal product structure with terminal unit
$1$, then natural transformations $F\Rightarrow P_n$ are determined
by their restrictions to tuples with all but one coordinate equal
to $1$.
\end{theorem}

\begin{proof}
Fix $\underline X=(X_1,\ldots,X_n)$. There is a morphism
\[
 a_s:E_s(X_s)\longrightarrow\underline X
\]
whose $s$-th component is $1_{X_s}$ and whose other components are
the unique maps from $0$. The canonical inclusion is
\[
 \iota_s=S_n(a_s)\kappa_{s,X_s}^{-1}.
\]
Naturality of $\theta$ gives the commutative diagram
\[
\xymatrix@C=3.5em@R=3em{
 &S_n(E_s(X_s))\ar[r]^{\theta_{E_s(X_s)}}\ar[d]_{S_n(a_s)}
 &F(E_s(X_s))\ar[d]^{F(a_s)}\\
 X_s\ar[ur]^{\kappa_{s,X_s}^{-1}}\ar[r]_{\iota_s}
 &S_n(\underline X)\ar[r]_{\theta_{\underline X}}
 &F(\underline X).
}
\]
Consequently,
\begin{equation}\label{newcoh:eq:restriction}
 \theta_{\underline X}\iota_s
 =F(a_s)\theta_{E_s(X_s)}\kappa_{s,X_s}^{-1}.
\end{equation}
The same formula holds for $\psi$. Agreement on the indicated
tuples therefore gives agreement after every inclusion $\iota_s$.
These inclusions are jointly epimorphic, so
$\theta_{\underline X}=\psi_{\underline X}$.
\end{proof}

At length zero, the corresponding uniqueness assertions follow
directly from initiality of $S_0=0$ and terminality of $P_0=1$.
The theorem does not assert that any of the natural transformations
under consideration exist, and does not require the sum or product
structure to be a categorical coproduct or product.

\begin{definition}\label{diag:def:prelinear}
A \emph{sum-product category} is a category equipped with a specified
monoidal sum structure $(\oplus,0)$ and a specified monoidal product
structure $(\otimes,1)$. No comparison between the two structures is
part of this definition, and the initial object $0$ need not be
terminal.
\end{definition}

For a row of morphisms $f_s:X_s\to Y$ $(1\leq s\leq n)$, a
\emph{realisation} is a morphism
\[
 f:X_1\oplus\cdots\oplus X_n\longrightarrow Y
 \qquad\text{such that}\qquad f\iota_s=f_s.
\]
It is unique when it exists, because the inclusions are jointly
epimorphic. Dually, a column of morphisms $g_r:X\to Y_r$
$(1\leq r\leq m)$ is realised by a morphism
\[
 g:X\longrightarrow Y_1\otimes\cdots\otimes Y_m
 \qquad\text{such that}\qquad \pi_rg=g_r,
\]
and its realisation is unique because the projections are jointly
monomorphic. The empty row and empty column are realised by the unique
maps out of $0$ and into $1$, respectively. Thus the sum structure
specifies which rows can be realised, while the product structure
specifies which columns can be realised. Neither structure requires
all such realisations to exist.

In a sum-product category, any morphism $f$ from an $n$-fold sum to an
$m$-fold product has an $m\times n$ matrix presentation
\[
\left[ \begin{array}{cccc}
\pi_1 f \iota_1 & \pi_1 f \iota_2 & \ldots & \pi_1 f \iota_n \\
\pi_2 f \iota_1 & \pi_2 f \iota_2 & \ldots & \pi_2 f \iota_n \\
\vdots & \vdots & & \vdots \\
\pi_m f \iota_1 & \pi_m f \iota_2 & \ldots & \pi_m f \iota_n
\end{array} \right].
\]
\begin{definition}\label{diag:def:matrix-realisation}
Let $X_1,\ldots,X_n$ and $Y_1,\ldots,Y_m$ be finite ordered lists of
objects in a sum-product category. A \emph{realisation} of a matrix
$(f_{rs}:X_s\to Y_r)_{r,s}$ is a morphism
\[
 f:\bigoplus_{s=1}^{n}X_s\longrightarrow\bigotimes_{r=1}^{m}Y_r
 \qquad\text{such that}\qquad \pi_rf\iota_s=f_{rs}
 \quad(1\leq r\leq m,\ 1\leq s\leq n).
\]
The empty sum is $0$ and the empty product is $1$; if either list is
empty, the component equations are vacuous.
\end{definition}

Now, unlike the case for coproducts and products, a given such matrix may not have a morphism it represents. But, when such a morphism exists, it is unique. This follows from the fact that inclusions are jointly epimorphic and projections are jointly monomorphic.

Rows index target factors and columns index source summands. We use the notation
$M(f)_{rs}=\pi_r f\iota_s$. Matrix uniqueness does not imply that an
arbitrary matrix is realised, nor does it supply a rule for adding or
multiplying matrices before the necessary morphisms have been constructed.

\begin{lemma}[A matrix as a row of columns or a column of rows]
\label{diag:prop:row-column}
For a matrix $(f_{rs}:X_s\to Y_r)_{r,s}$ in a sum-product category,
the following are equivalent:
\begin{enumerate}
\item The matrix has a realisation
$f:\bigoplus_sX_s\to\bigotimes_rY_r$.
\item Every column has a realisation
$c_s:X_s\to\bigotimes_rY_r$, and the row $(c_s)_s$ has a
realisation.
\item Every row has a realisation
$d_r:\bigoplus_sX_s\to Y_r$, and the column $(d_r)_r$ has a
realisation.
\end{enumerate}
When these conditions hold, the realisations in the three assertions
are the same morphism.
\end{lemma}
\begin{proof}
If $f$ realises the matrix, take $c_s=f\iota_s$ and
$d_r=\pi_rf$. Their respective components are $f_{rs}$, and $f$
is both the realisation of the row $(c_s)_s$ and the realisation of
the column $(d_r)_r$. Conversely, a realisation $h$ of the row of
column realisations satisfies
$\pi_rh\iota_s=\pi_rc_s=f_{rs}$. A realisation $k$ of the
column of row realisations satisfies
$\pi_rk\iota_s=d_r\iota_s=f_{rs}$. Matrix uniqueness gives
$h=k=f$ whenever any of these morphisms is given. The same argument
covers empty lists, using the initial and terminal objects.
\end{proof}

Suppose now that $\C$ is a pointed sum-product category. Its units
$0$ and $1$ are zero objects, and the unique map $j:0\to1$ is an
isomorphism. They need not be the same specified object.

\begin{theorem}\label{newcoh:natural-matrices}
Let $n\geq1$. Every natural transformation
$\theta:S_n\Rightarrow P_n$ determines unique natural
endomorphisms $\alpha_s:\mathrm{Id}_{\C}\Rightarrow
\mathrm{Id}_{\C}$, for $1\leq s\leq n$, such that
\begin{equation}\label{newcoh:eq:natural-matrix}
 \pi_r\theta_{\underline X}\iota_s
 =\begin{cases}
    \alpha_{s,X_s},&r=s,\\
    0,&r\ne s.
  \end{cases}
\end{equation}
The transformation $\theta$ is uniquely determined by this family.
\end{theorem}

\begin{proof}
Define
\[
 \alpha_{s,X}
 =\pi_s\theta_{E_s(X)}\kappa_{s,X}^{-1}.
\]
Here the $s$-th projection from $P_n(E_s(X))$ is an isomorphism:
all its other factors are zero objects, hence are isomorphic to the
product unit. The displayed composites are natural in $X$, so each
$\alpha_s$ is a natural endomorphism of the identity functor.

Apply the $r$-th projection to
\eqref{newcoh:eq:restriction}, taking $F=P_n$. Naturality of the
projections gives
\[
 \pi_r\theta_{\underline X}\iota_s
 =(a_s)_r\,\pi_r\theta_{E_s(X_s)}\kappa_{s,X_s}^{-1}.
\]
For $r=s$, the component $(a_s)_s$ is an identity, giving
$\alpha_{s,X_s}$. For $r\ne s$, the composite factors through
the zero object in position $r$ of $E_s(X_s)$, so it is zero.
This proves \eqref{newcoh:eq:natural-matrix}. Matrix uniqueness
then determines $\theta$ from the $\alpha_s$, and evaluation on
$E_s(X)$ recovers each $\alpha_s$ uniquely.
\end{proof}

Thus naturality already forces the off-diagonal entries to vanish.

\subsection{Diagolinearity}

Write $j:0\to1$ for the unique morphism from the initial sum unit
to the terminal product unit. When $\C$ is pointed, both units are
zero objects and $j$ is an isomorphism. We denote the zero morphism
from $X$ to $Y$ by $z_{X,Y}$, or simply by $0$ when its source and
target are clear. Let $S_n(X_1,\ldots,X_n)$ and
$P_n(X_1,\ldots,X_n)$ denote right-associated $n$-fold sums and
products, with $S_0=0$, $P_0=1$ and $S_1(X)=P_1(X)=X$.

\begin{definition}\label{def:prelineariser}\label{diag:def:identity-comparison}
\label{diag:def:diagolinear}
A \emph{diagolinear category} is a pointed sum-product category in
which every finite diagonal matrix is realisable. Explicitly, for
every finite family of morphisms $f_s:X_s\to Y_s$ $(1\leq s\leq n)$,
there is a morphism
\[
 D(f_1,\ldots,f_n):S_n(X_1,\ldots,X_n)
       \longrightarrow P_n(Y_1,\ldots,Y_n)
\]
whose matrix entries are
\[
 \pi_rD(f_1,\ldots,f_n)\iota_s
 =\begin{cases}f_s,&r=s,\\0,&r\ne s.\end{cases}
\]
The realisation is unique by matrix uniqueness. At length zero it
is $j:0\to1$, and at length one it is the given morphism.

In a pointed sum-product category, a family
$i_{A,B}:A\oplus B\to A\otimes B$ realising the binary identity
matrices
\[
 \begin{pmatrix}1_A&z_{B,A}\\z_{A,B}&1_B\end{pmatrix}
\]
is called an \emph{identity comparison}. Naturality is not part of
this definition. Every diagolinear category has a unique identity
comparison.
\end{definition}

The definition can be reduced to binary matrices. There are also
equivalent conditions which do not assume pointedness; the following
observation explains why pointedness nevertheless follows.

\begin{lemma}\label{prop:pointed}

Let $\C$ be a sum-product category. If there is a morphism
$i_{1,0}:1\oplus0\to1\otimes0$, then $\C$ is pointed.
In particular, this holds whenever there is a natural transformation
$i:\oplus\rightarrow\otimes$.

\end{lemma}

\begin{proof}

    The following composite is a morphism from 1 to 0.
    \[\xymatrix{
        1 \ar[r]^-{\rho^{-1}_{\oplus}} & 1 \oplus 0 \ar[r]^-{i_{1,0}} & 1 \otimes 0 \ar[r]^-{\lambda_{\otimes}} & 0
    }\]
    It is inverse to the unique map $j:0\to1$, since both composites
    are endomorphisms of an initial or a terminal object. Thus $0$ and
    $1$ are both zero objects.
\end{proof}

\begin{theorem}[Characterisations by matrix entries]
\label{thm:identity-matrix}
\label{Thoerem: monoidal-sums: transformer iff identity matrix}
\label{diag:prop:identity-patterns}
For a sum-product category $\C$, the following are equivalent.
\begin{enumerate}
\item $\C$ is diagolinear.
\item $\C$ is pointed and, for every pair $A,B$, the binary identity
matrix has a realisation $i_{A,B}:A\oplus B\to A\otimes B$.
\item For every pair $A,B$, there is a morphism
$h_{A,B}:A\oplus B\to A\otimes B$ such that
\[
 \pi_1h_{A,B}\iota_1=1_A,\qquad
 \pi_2h_{A,B}\iota_2=1_B,
\]
and each of $\pi_1h_{A,B}\iota_2$ and $\pi_2h_{A,B}\iota_1$
factors through either the initial object $0$ or the terminal object
$1$. The choice may be made independently for each entry and pair.
No naturality is assumed.
\item There is a natural transformation $h:\oplus\Rightarrow\otimes$
whose components satisfy
\[
 \pi_1h_{A,B}\iota_1=1_A,\qquad
 \pi_2h_{A,B}\iota_2=1_B
 \qquad\text{for every }A,B.
\]
No condition on the off-diagonal entries is assumed.
\item There is a natural transformation $h:\oplus\Rightarrow\otimes$
whose components satisfy only the unit-entry conditions
\[
 \pi_1h_{A,0}\iota_1=1_A,\qquad
 \pi_2h_{0,A}\iota_2=1_A
 \qquad\text{for every }A.
\]
\end{enumerate}
In \textup{(2)}--\textup{(5)}, the indicated binary family is the unique
identity comparison. In particular, every identity comparison is
automatically natural.
\end{theorem}
\begin{proof}
The implication \textup{(1)}$\Rightarrow$\textup{(2)} follows by
taking binary diagonal matrices with identity diagonal entries.
For the converse, suppose \textup{(2)} holds. We construct a
realisation of an arbitrary finite diagonal matrix by induction.
Put $D()=j$ and $D(f_1)=f_1$. For $n\geq2$, define
\begin{equation}\label{diag:eq:recursive-diagonal}
 D(f_1,\ldots,f_n)
 =\bigl(f_1\otimes D(f_2,\ldots,f_n)\bigr)
    i_{X_1,S_{n-1}(X_2,\ldots,X_n)}.
\end{equation}
Regard the source as $X_1\oplus S_{n-1}$ and the target as
$Y_1\otimes P_{n-1}$. The first column of the binary identity
comparison has first component $1_{X_1}$ and second component zero.
After the displayed postcomposition, its matrix entries are $f_1$
followed by zeros. The second binary column has first component
zero and second component $1_{S_{n-1}}$. Restricting this column
to each inclusion of $X_s$ for $s>1$ therefore gives first entry
zero and the corresponding column of $D(f_2,\ldots,f_n)$ below it.
By induction these are precisely the remaining columns of the
required diagonal matrix. Thus every finite diagonal matrix is
realisable, proving \textup{(1)}. This construction uses neither
naturality nor any additional compatibility of the binary
realisations.

Condition \textup{(2)} implies \textup{(3)}, since zero morphisms
factor through either zero object. Conversely, under \textup{(3)}
the existence of $h_{1,0}$ makes $\C$ pointed by
Lemma~\ref{prop:pointed}. Both $0$ and $1$ are then zero objects,
so every off-diagonal entry specified in \textup{(3)} is zero.
Each $h_{A,B}$ consequently realises the binary identity matrix.
This proves \textup{(2)}.

Under \textup{(2)}, matrix uniqueness gives a unique $i_{A,B}$ for
each pair. Given $f:A\to A'$ and $g:B\to B'$, the two composites
in the square
\[
\xymatrix@C=4em@R=3em{
 A\oplus B\ar[r]^{f\oplus g}\ar[d]_{i_{A,B}}
 &A'\oplus B'\ar[d]^{i_{A',B'}}\\
 A\otimes B\ar[r]_{f\otimes g}&A'\otimes B'
}
\]
have the same matrix
\[
 \begin{pmatrix}f&0\\0&g\end{pmatrix}.
\]
They are therefore equal, proving naturality and hence
\textup{(4)}. Taking the pairs $(A,0)$ and $(0,A)$ gives
\textup{(4)}$\Rightarrow$\textup{(5)}.

Finally, assume \textup{(5)}. The component $h_{1,0}$ again makes
$\C$ pointed. Naturality of $h$, of the canonical inclusions and
of the canonical projections, applied to
$(1_A,0\to B):(A,0)\to(A,B)$ and
$(0\to A,1_B):(0,B)\to(A,B)$, gives
\[
 \pi_1h_{A,B}\iota_1=\pi_1h_{A,0}\iota_1=1_A,
 \qquad
 \pi_2h_{A,B}\iota_2=\pi_2h_{0,B}\iota_2=1_B.
\]
The zero endomorphism $z_{B,B}$ is now available. Naturality at
$(1_A,z_{B,B})$ says that
\[
 (1_A\otimes z_{B,B})h_{A,B}
       =h_{A,B}(1_A\oplus z_{B,B}).
\]
Taking the top-right matrix entry gives
$\pi_1h_{A,B}\iota_2=0$, and taking the bottom-left entry gives
$0=\pi_2h_{A,B}\iota_1$. Thus every component has the binary
identity matrix, proving \textup{(2)}. 
Matrix uniqueness also
shows that all the families in the statement are the same identity
comparison.
\end{proof}

Thus pointedness and naturality can both be omitted from the
hypotheses when the diagonal entries are identities and the
off-diagonal entries have the factorisations in \textup{(3)}.
Alternatively, naturality allows the off-diagonal conditions to
be omitted and the diagonal conditions to be tested only at the
units, as in \textup{(5)}.

The unit equations for an identity comparison are consequences of
its matrix entries. The following formulation also relates the
unit-entry characterisation to the monoidal unitors.

\begin{lemma}\label{compat from i}
Let $i:\oplus\Rightarrow\otimes$ be a natural transformation in
a sum-product category. Then $i$ is an identity comparison if and
only if
\[
 \pi_1i_{A,0}\iota_1=1_A,\qquad
 \pi_2i_{0,A}\iota_2=1_A
 \qquad\text{for every }A.
\]
The first equality is equivalent to commutativity of the right
wing of the following butterfly, and the second to commutativity
of its left wing:
\[
\xymatrix@C=4pc@R=3pc{
  0\oplus A
    \ar[d]_{i_{0,A}}
    \ar[rd]^{\lambda_{\oplus,A}}
  &&
  A\oplus0
    \ar[d]^{i_{A,0}}
    \ar[ld]_{\rho_{\oplus,A}}
  \\
  0\otimes A
    \ar[d]_{j\otimes1_A}
  & A &
  A\otimes0
    \ar[d]^{1_A\otimes j}
  \\
  1\otimes A
    \ar[ur]_{\lambda_{\otimes,A}}
  &&
  A\otimes1
    \ar[lu]^{\rho_{\otimes,A}}
}
\]
We call these equations compatibility with $\rho$ and $\lambda$,
respectively. In particular, both equations hold for the identity
comparison of every diagolinear category.
\end{lemma}
\begin{proof}
The first equivalence is Theorem~\ref{thm:identity-matrix}.
The existence of $i$ makes $\C$ pointed, so $j$ is an isomorphism.
For the pair $(A,0)$ the first inclusion is
$\rho_{\oplus,A}^{-1}$ and the first projection is
$\rho_{\otimes,A}(1_A\otimes j)$. Hence
\[
 \pi_1i_{A,0}\iota_1
 =\rho_{\otimes,A}(1_A\otimes j)i_{A,0}
     \rho_{\oplus,A}^{-1}.
\]
This is $1_A$ exactly when
\[
 \rho_{\otimes,A}(1_A\otimes j)i_{A,0}=\rho_{\oplus,A},
\]
which is commutativity of the right wing. For the pair $(0,A)$,
the second inclusion is $\lambda_{\oplus,A}^{-1}$ and the second
projection is $\lambda_{\otimes,A}(j\otimes1_A)$. Thus
\[
 \pi_2i_{0,A}\iota_2
 =\lambda_{\otimes,A}(j\otimes1_A)i_{0,A}
     \lambda_{\oplus,A}^{-1}
\]
is $1_A$ exactly when
\[
 \lambda_{\otimes,A}(j\otimes1_A)i_{0,A}=\lambda_{\oplus,A},
\]
which is commutativity of the left wing.
\end{proof}

The identity comparisons also determine the realisations of all
diagonal matrices in a way that is compatible with the two monoidal
structures.

\begin{theorem}\label{diag:prop:diagonal-realisations}
In a diagolinear category, let $I_{X_1,\ldots,X_n}$ be the unique
realisation of the identity matrix on $X_1,\ldots,X_n$. For any
morphisms $f_s:X_s\to Y_s$ $(1\leq s\leq n)$, the two composites
in the following square are equal and give the unique realisation
$D(f_1,\ldots,f_n)$ of the corresponding diagonal matrix:
\[
\xymatrix@C=4em@R=3em{
 S_n(X_1,\ldots,X_n)\ar[r]^{S_n(f_1,\ldots,f_n)}
      \ar[d]_{I_{X_1,\ldots,X_n}}
 &S_n(Y_1,\ldots,Y_n)\ar[d]^{I_{Y_1,\ldots,Y_n}}\\
 P_n(X_1,\ldots,X_n)\ar[r]_{P_n(f_1,\ldots,f_n)}
 &P_n(Y_1,\ldots,Y_n).
}
\]
Here $I_0=j$, $I_X=1_X$, and the binary components are $i_{X,Y}$.
Other parenthesisations and inserted units give the same
realisations under the canonical monoidal identifications.
\end{theorem}
\begin{proof}
For the composite through the top and right edges, naturality of
the sum inclusions gives
\[
 \pi_r I_{Y_1,\ldots,Y_n}
       S_n(f_1,\ldots,f_n)\iota_s
 =\pi_r I_{Y_1,\ldots,Y_n}\iota_s f_s
 =\begin{cases}f_s,&r=s,\\0,&r\ne s.\end{cases}
\]
For the composite through the left and bottom edges, naturality
of the product projections similarly gives
\[
 \pi_r P_n(f_1,\ldots,f_n)
       I_{X_1,\ldots,X_n}\iota_s
 =f_r\pi_r I_{X_1,\ldots,X_n}\iota_s
 =\begin{cases}f_s,&r=s,\\0,&r\ne s.\end{cases}
\]
Matrix uniqueness proves both the equality and the stated
identification. At length zero both composites are $j$, and at
length one both are $f_1$. Canonical monoidal identifications
preserve inclusions and projections, so the assertion about other
representatives follows from the same matrix calculation.
\end{proof}

\begin{theorem}[Normalisation of an invertible identity comparison]
\label{diag:thm:normalisation}
Let $\C$ be a diagolinear category with sum structure
$(\oplus,0)$, product structure $(\otimes,1)$ and invertible
identity comparison $i$. Then $(\oplus,0)$ can be used as both
the sum structure and the product structure on $\C$, with the
identity natural transformation as identity comparison. The
inclusions remain $\iota_1,\iota_2$, and the projections of the
new product structure are
\begin{equation}\label{diag:eq:normalised-projections}
 p_r=\pi_r i_{X,Y}:X\oplus Y\longrightarrow X_r
 \qquad(r=1,2),
\end{equation}
where $X_1=X$ and $X_2=Y$. The resulting sum-product category is equivalent to the original
one by a functor whose underlying functor is the identity and
which is strong monoidal for both structures. Dually, one may
use $(\otimes,1)$ as both structures.
\end{theorem}

\begin{proof}
Since $\C$ is pointed, its specified initial object $0$ is also
terminal. The monoidal structure $(\oplus,0)$ therefore has
canonical projections
\[
 p_1=\rho_{\oplus,X}(1_X\oplus {!_Y}),
 \qquad
 p_2=\lambda_{\oplus,Y}({!_X}\oplus1_Y),
\]
where ${!_X}:X\to0$ and ${!_Y}:Y\to0$ are the unique arrows.
Together with the canonical inclusions of the sum structure,
these satisfy
\[
 p_r\iota_s=\delta_{rs}.
\]
Indeed, the diagonal composites are identities by the unit
equations, while each off-diagonal composite factors through
$0\oplus0$, which is a zero object.

The identity comparison satisfies
\[
 (\pi_r i_{X,Y})\iota_s=\delta_{rs}
\]
as well. Since the inclusions are jointly epimorphic, this proves
\eqref{diag:eq:normalised-projections}. The original projections
$\pi_1,\pi_2$ are jointly monomorphic, and $i_{X,Y}$ is
invertible. Consequently $p_1,p_2$ are jointly monomorphic.
Thus $(\oplus,0)$ is also a product structure.

The identity of $X\oplus Y$ has identity matrix with respect to
the inclusions $\iota_s$ and the projections $p_r$. More generally,
each diagonal matrix is realised by the corresponding iterated
monoidal product of its diagonal entries. Hence the new
sum-product category is diagolinear and its identity comparison
is the identity natural transformation.

It remains to check that this change of structures is compatible
with their original associators and unitors. Write
$j:0\to1$ for the unique arrow, which is an isomorphism because
both objects are zero objects. The identity functor
\[
 (\C,\otimes,1)\longrightarrow(\C,\oplus,0)
\]
has strong monoidal structure maps $i_{X,Y}$ and $j$. For
associativity, the required diagram is
\[
\xymatrix@C=4em@R=3em{
 (A\oplus B)\oplus C
 \ar[r]^{\alpha_\oplus^{-1}}
 \ar[d]_{i_{A,B}\oplus1_C}
 &
 A\oplus(B\oplus C)
 \ar[d]^{1_A\oplus i_{B,C}}
 \\
 (A\otimes B)\oplus C
 \ar[d]_{i_{A\otimes B,C}}
 &
 A\oplus(B\otimes C)
 \ar[d]^{i_{A,B\otimes C}}
 \\
 (A\otimes B)\otimes C
 \ar[r]_{\alpha_\otimes^{-1}}
 &
 A\otimes(B\otimes C).
}
\]
Both composites from the top left object to the bottom right
object have the $3\times3$ identity matrix: their diagonal entries
are $1_A,1_B,1_C$, and their off-diagonal entries are zero.
This follows from the identity matrices of the components of
$i$ and the preservation of the ordered coordinate maps by the
associators. The iterated sum inclusions are jointly epimorphic
and the iterated product projections are jointly monomorphic,
so matrix uniqueness proves commutativity.

For the right unit condition, naturality of $i$ and
\eqref{diag:eq:normalised-projections} give
\[
\begin{aligned}
 \rho_{\otimes,A} i_{A,1}(1_A\oplus j)
 &=
 \rho_{\otimes,A}(1_A\otimes j)i_{A,0}\\
 &=
 \pi_1i_{A,0}
 =
 \rho_{\oplus,A}.
\end{aligned}
\]
Likewise, the left unit condition is
\[
\begin{aligned}
 \lambda_{\otimes,A} i_{1,A}(j\oplus1_A)
 &=
 \lambda_{\otimes,A}(j\otimes1_A)i_{0,A}\\
 &=
 \pi_2i_{0,A}
 =
 \lambda_{\oplus,A}.
\end{aligned}
\]
Thus the identity functor, with these structure maps on the
product side and identity structure maps on the sum side,
gives the asserted equivalence. Its inverse uses $i^{-1}$ and
$j^{-1}$. The corresponding assertion using $(\otimes,1)$
for both structures follows by duality.
\end{proof}

\begin{theorem}\label{diag:cor:natural-diagonals}
Let $\C$ be a diagolinear category and let $n\geq1$.
For every family of natural endomorphisms
$\alpha_s:\mathrm{Id}_{\C}\Rightarrow\mathrm{Id}_{\C}$
$(1\leq s\leq n)$, there is a unique natural transformation
$\theta:S_n\Rightarrow P_n$ whose components have matrix
\[
 \pi_r\theta_{\underline X}\iota_s
 =
 \begin{cases}
   \alpha_{s,X_s},&r=s,\\
   0,&r\ne s.
 \end{cases}
\]
Consequently, taking diagonal entries gives a bijection
\[
 \operatorname{Nat}(S_n,P_n)
 \cong
 \operatorname{Nat}(\mathrm{Id}_{\C},\mathrm{Id}_{\C})^n.
\]
\end{theorem}

\begin{proof}
By Theorem~\ref{diag:prop:diagonal-realisations}, the morphism
\[
 \theta_{\underline X}
 =
 P_n(\alpha_{1,X_1},\ldots,\alpha_{n,X_n})
 I_{X_1,\ldots,X_n}
\]
is the unique realisation of the prescribed diagonal matrix.
The identity comparisons $I_{X_1,\ldots,X_n}$ are natural by
the same theorem. Naturality of the $\alpha_s$ and functoriality
of $P_n$ therefore show that the displayed formula defines a
natural transformation. Uniqueness holds componentwise by
matrix uniqueness.

By Theorem~\ref{newcoh:natural-matrices}, every natural
transformation $S_n\Rightarrow P_n$ has precisely this form
for a unique family $(\alpha_s)_{s=1}^{n}$, proving the
bijection.
\end{proof}

\subsection{Coherence theorem for diagolinear categories}

Let $\C$ be diagolinear and let
$i_{A,B}:A\oplus B\to A\otimes B$ be its natural identity
comparison. A \emph{mixed word} is an expression formed from the
symbols $0,1$, an ordered list of variables, and the operations
$\oplus,\otimes$, with parentheses. Every variable occurs exactly
once, in the prescribed order. Its length is its number of variable
occurrences. A pure sum word uses only $\oplus$ and its unit $0$;
a pure product word uses only $\otimes$ and its unit $1$.
Evaluation of a word of length $n$ gives a functor
$\C^n\to\C$.

The \emph{forward canonical transformations} are generated by the
associators and unitors of both structures and their inverses, the
maps $j,j^{-1}$, identities, and the components of $i$, under
composition and the two monoidal operations. The generators may be
applied to arbitrary subwords. We keep the same ordered variable
occurrences throughout. In particular, no permutations of variables
are admitted, and $i^{-1}$ is not a generator, even when it exists.

Every word $Z$ of length zero evaluates to a zero object. Indeed,
this follows by induction: it holds for $0$ and $1$, and the sum
or product of two zero objects is isomorphic, by the relevant
unitors, to the corresponding unit. The unique isomorphisms
\[
 z_Z^0:Z\longrightarrow0,\qquad
 z_Z^1:Z\longrightarrow1
\]
are canonical. To construct them, apply the induction to the two
subwords, use the relevant unitor, and then use $j$ or $j^{-1}$
if the other zero object is required. Uniqueness also gives
$z_Z^1=jz_Z^0$.

For a word $U$ with one variable $X$, define a natural canonical
isomorphism $c_U:U(X)\to X$ recursively. Put $c_X=1_X$.
If $\diamond$ is either monoidal operation, write
$e_{\oplus}=0$ and $e_{\otimes}=1$. Every nontrivial word with
one variable has exactly one subword containing that variable;
the other subword, denoted by $Z$, has length zero. Set
\begin{equation}\label{newcoh:eq:collapse}
 \begin{aligned}
 c_{U\diamond Z}
   &=\rho_{\diamond,X}(c_U\diamond z_Z^{e_{\diamond}}),\\
 c_{Z\diamond U}
   &=\lambda_{\diamond,X}(z_Z^{e_{\diamond}}\diamond c_U).
 \end{aligned}
\end{equation}
These maps are natural isomorphisms, constructed without using $i$.

\begin{lemma}[One-variable normalisation]
\label{newcoh:normalisation}
For every forward canonical transformation $h:U\Rightarrow V$
between words with one variable, the following square commutes:
\[
\xymatrix@C=4em@R=3em{
 U(X)\ar[r]^{h_X}\ar[d]_{c_U}
 &V(X)\ar[d]^{c_V}\\
 X\ar[r]_{1_X}&X.
}
\]
For words of length zero, every such transformation is the unique
map between their zero-object values.
\end{lemma}

\begin{proof}
The length-zero assertion follows from the preceding observation.
For length one, we check the generators and the operations used to
form canonical transformations.

Consider an associator for one of the operations $\diamond$.
Exactly one of its three argument words contains the variable.
Use its map $c_U$ to replace that argument by $X$, and use the
unique maps $z_Z^{e_{\diamond}}$ to replace the other two arguments
by the unit $e_{\diamond}$. Naturality of the associator reduces
the required equation to one for the three arguments
\[
 (X,e_{\diamond},e_{\diamond}),\quad
 (e_{\diamond},X,e_{\diamond}),\quad\text{or}\quad
 (e_{\diamond},e_{\diamond},X).
\]
The two paths to $X$ agree by ordinary monoidal coherence. If an
intermediate subword has length zero, its recursively constructed
map to $e_{\diamond}$ is the unique such map, so it agrees with
the one used in this reduction. This proves the assertion for all
three positions of the variable. The same reduction applies to
unitors. Their inverses and the inverse associators satisfy the
equation by reversing the corresponding invertible arrows.

It remains to check an identity comparison with one variable.
Suppose its arguments are $U$ and a zero word $Z$. The first
inclusion $\iota_1:U\to U\oplus Z$ and first projection
$\pi_1:U\otimes Z\to U$ are isomorphisms. The identity-matrix
equation $\pi_1i_{U,Z}\iota_1=1_U$ therefore gives the first
square below. The second follows in the same way from
$\pi_2i_{Z,U}\iota_2=1_U$:
\[
\xymatrix@C=3em@R=3em{
 U\oplus Z\ar[r]^{i_{U,Z}}\ar[d]_{\iota_1^{-1}}
 &U\otimes Z\ar[d]^{\pi_1}
 & Z\oplus U\ar[r]^{i_{Z,U}}\ar[d]_{\iota_2^{-1}}
 &Z\otimes U\ar[d]^{\pi_2}\\
 U\ar[r]_{1_U}&U&U\ar[r]_{1_U}&U.
}
\]
By \eqref{newcoh:eq:collapse} and naturality of the unitors,
\[
 \begin{aligned}
 c_{U\oplus Z}=c_U\iota_1^{-1},&\qquad
 c_{U\otimes Z}=c_U\pi_1,\\
 c_{Z\oplus U}=c_U\iota_2^{-1},&\qquad
 c_{Z\otimes U}=c_U\pi_2.
 \end{aligned}
\]
Composing the two squares with $c_U$ proves the required equations
for $i$. These are the unit-compatibility equations expressed in
terms of matrix entries. If both arguments have length zero, the
length-zero assertion applies. The generators $j,j^{-1}$ themselves
have length zero, and identities satisfy the assertion immediately.

The assertion is preserved by composition. To check preservation
under either operation $\diamond$, suppose
$f:U\to U'$ has one variable and $g:Z\to Z'$ has length zero.
Using the induction hypothesis and uniqueness of maps between
zero objects, we obtain
\[
 \begin{aligned}
 c_{U'\diamond Z'}(f\diamond g)
 &=\rho_{\diamond,X}
       (c_{U'}f\diamond z_{Z'}^{e_{\diamond}}g)\\
 &=\rho_{\diamond,X}
       (c_U\diamond z_Z^{e_{\diamond}})
  =c_{U\diamond Z}.
 \end{aligned}
\]
When the variable is in the second factor, the same calculation
uses $\lambda_{\diamond,X}$. When both factors have length zero,
uniqueness applies. Structural induction now proves the assertion
for every forward canonical transformation, including occurrences
of the generators inside arbitrary mixed words.
\end{proof}

\begin{theorem}[Coherence for diagolinear categories]
\label{newcoh:coherence}
Let $\C$ be diagolinear. Fix an ordered list of $n$ variable
occurrences, a pure sum word $s$, a pure product word $p$, and a
mixed word $w$ on that list. There is exactly one forward canonical
natural transformation $s\Rightarrow w$, and exactly one forward
canonical natural transformation $w\Rightarrow p$. In particular, there is exactly one forward canonical natural
transformation $s\Rightarrow p$. At every tuple of objects, its
matrix is the $n\times n$ identity matrix. Thus all forward
canonical paths from $s$ to $p$ agree, including paths with mixed
intermediate words.
\end{theorem}

\begin{proof}
We first construct the transformations. By ordinary monoidal
coherence, it is enough to use the standard words $S_n$ and $P_n$
as the pure endpoints. For every mixed word $w$ of length $n$, we
construct forward canonical transformations
\[
 q_w:S_n\Rightarrow w,\qquad r_w:w\Rightarrow P_n
\]
by induction. For a variable, both are identities. At length-zero
generators, put
\[
 q_0=1_0,\quad r_0=j,\qquad q_1=j,\quad r_1=1_1.
\]
Suppose $u,v$ have lengths $m,k$. Write
\[
 a_{m,k}:S_{m+k}\xrightarrow{\cong}S_m\oplus S_k,
 \qquad
 b_{m,k}:P_m\otimes P_k\xrightarrow{\cong}P_{m+k}
\]
for the canonical monoidal identifications, using the corresponding
blocks of variables. These also exist when a block is empty.
Define
\[
 \begin{aligned}
 q_{u\oplus v}&=(q_u\oplus q_v)a_{m,k},\\
 q_{u\otimes v}&=(q_u\otimes q_v)i_{S_m,S_k}a_{m,k},\\
 r_{u\oplus v}&=b_{m,k}(r_u\otimes r_v)i_{u,v},\\
 r_{u\otimes v}&=b_{m,k}(r_u\otimes r_v).
 \end{aligned}
\]
Every factor is forward canonical, proving existence.

For uniqueness of $S_n\Rightarrow w$, assume first that $n\geq1$.
Let $h$ be any such forward canonical transformation. Substituting
$E_s(X)$ for its variables gives a canonical transformation between
words with one surviving variable. Lemma~\ref{newcoh:normalisation}
gives
\[
 c_{wE_s}\,h_{E_s(X)}=\kappa_{s,X},
 \qquad\text{and hence}\qquad
 h_{E_s(X)}=c_{wE_s}^{-1}\kappa_{s,X}.
\]
Here $c_{wE_s}:w(E_s(X))\to X$ is the recursively defined
unit-cancellation isomorphism. The right-hand side is independent
of $h$. Theorem~\ref{newcoh:restriction} therefore proves
uniqueness. At length zero, the source $S_0=0$ is initial, so
uniqueness is immediate.

For uniqueness of $w\Rightarrow P_n$, apply the dual part of
Theorem~\ref{newcoh:restriction}. At a tuple $T_s(X)$ with $1$
in all other positions, there is again exactly one surviving
variable, since $1$ is a zero object. The normalisation lemma gives
\[
 \nu_{s,X}\,h_{T_s(X)}=c_{wT_s},
\]
which determines the restricted component. At length zero, use
terminality of $P_0=1$. Conjugating by the ordinary monoidal
identifications proves the assertions for arbitrary pure words
$s$ and $p$.

Finally, let $h:S_n\Rightarrow P_n$ be canonical. On $E_s(X)$,
the collapse map for $S_n$ is $\kappa_{s,X}$, and the collapse map
for $P_n$ is its $s$-th canonical projection. Thus
\[
 \pi_s h_{E_s(X)}\kappa_{s,X}^{-1}=1_X.
\]
Theorem~\ref{newcoh:natural-matrices} gives identity diagonal
entries and zero off-diagonal entries for every component of $h$.
Canonical identifications of pure words preserve inclusions and
projections, so the same matrix statement holds for $s\Rightarrow p$.
For $n=0$, the unique transformation is $j$, realising the empty
identity matrix.
\end{proof}

The transformation $S_n\Rightarrow P_n$ is therefore the unique
realisation of the identity matrix, independently of parentheses,
inserted units, and intermediate mixed words. For a family of
arrows $f_s:X_s\to Y_s$, composing it with $P_n(f_1,\ldots,f_n)$
gives the unique realisation of the corresponding diagonal matrix.
No inverse of the identity comparison is used. The argument does
not assert uniqueness of canonical transformations between two
arbitrary mixed words: the determination theorem requires a pure
sum source or a pure product target.

Ordinary monoidal coherence controls the associators and unitors
within each of the two structures~\cite[Chapter~VII]{maclane_1998}.
It enters the proof above in the canonical identifications of pure
words and in the one-variable normalisation of associators and
unitors. The additional step involving $i$ follows from its identity
matrix. The passage from one variable to arbitrarily many variables
is supplied by Theorem~\ref{newcoh:restriction}, using naturality
and the jointly epimorphic inclusions, or dually the jointly
monomorphic projections.

Coherence for categorical products and coproducts has also been
studied by proof-theoretic methods. Cockett and
Seely~\cite{CockettSeely2001} give a deductive system for categories
with finite products and coproducts, obtain cut elimination and
decidability of equality of morphisms, and establish a Whitman
theorem for the free such categories. Do\v{s}en and
Petri\'{c}~\cite{DosenPetric2002} prove a coherence theorem for
binary products and coproducts without their nullary units, by a
faithful interpretation of the free category in relations on finite
ordinals. Their later survey~\cite{DosenPetric2009} treats variants
of bicartesian coherence and corrects some earlier formulations and
proofs; the presence of units requires particular care.

These results provide a related study of equality of structural
morphisms, but do not directly yield Theorem~\ref{newcoh:coherence}.
Their product and coproduct operations have universal properties,
whereas the structures used here need only have jointly epimorphic
inclusions and jointly monomorphic projections. Conversely, the
theorem above concerns a specified forward canonical calculus with
fixed ordered variable occurrences and the stated restrictions on
its endpoints. It is not a decision procedure for all morphisms in
a free bicartesian category.

\subsection{Duoidal sum-product categories}
\label{duo:sec:interchange}

The interchange transformation of a duoidal category \cite{Garner2009,AguiarMahajan2010,BataninMarkl2012,Street2012,BookerStreet2013} can be studied
already in a sum-product category, before an identity comparison has
been given.

A \emph{duoidal structure} on a sum-product category consists of a natural family
\[
 \gamma_{A,B,C,D}:(A\otimes B)\oplus(C\otimes D)
       \longrightarrow(A\oplus C)\otimes(B\oplus D)
\]
and the unit structure maps
\begin{equation}\label{duo:eq:unit-maps}
 \delta:0\longrightarrow0\otimes0,\qquad
 \mu:1\oplus1\longrightarrow1,\qquad
 \tau:0\longrightarrow1,
\end{equation}
subject to the following conditions. The unit structure maps are
necessarily the unique maps with these sources and targets, by
initiality of $0$ and terminality of $1$.

For six objects $A_1,A_2,A_3,B_1,B_2,B_3$, put
\[
\begin{aligned}
 T_r&=A_r\otimes B_r,& U_r&=A_r\oplus B_r,\\
 P_{rs}&=(A_r\oplus A_s)\otimes(B_r\oplus B_s),&
 Q_{rs}&=(A_r\otimes A_s)\oplus(B_r\otimes B_s),
\end{aligned}
\]
and write
\[
\begin{aligned}
 H&=T_1\oplus T_2\oplus T_3,
 \qquad W=U_1\otimes U_2\otimes U_3,\\
 K&=(A_1\oplus A_2\oplus A_3)\otimes
       (B_1\oplus B_2\oplus B_3),\\
 V&=(A_1\otimes A_2\otimes A_3)\oplus
       (B_1\otimes B_2\otimes B_3).
\end{aligned}
\]
The two associativity conditions are commutativity of the diagrams
\begin{equation}\label{duo:eq:associativity}
\vcenter{\hbox{$
\xymatrix@C=3em@R=3em{
 H\ar[r]^{\gamma\oplus1}\ar[d]_{1\oplus\gamma}
 &P_{12}\oplus T_3\ar[d]^{\gamma}\\
 T_1\oplus P_{23}\ar[r]_{\gamma}&K
}
$}}
\qquad
\vcenter{\hbox{$
\xymatrix@C=3em@R=3em{
 V\ar[r]^{\gamma}\ar[d]_{\gamma}
 &Q_{12}\otimes U_3\ar[d]^{\gamma\otimes1}\\
 U_1\otimes Q_{23}\ar[r]_{1\otimes\gamma}&W.
}
$}}
\end{equation}
Each occurrence of $\gamma$ has the arguments specified by its source
and target; the canonical reassociations are understood. Writing the
unitors with the relevant operation as a subscript, the four unit
conditions are
\begin{equation}\label{duo:eq:units}
\begin{aligned}
 (\rho_{\oplus,A}\otimes\rho_{\oplus,B})
 \gamma_{A,B,0,0}(1_{A\otimes B}\oplus\delta)
     &=\rho_{\oplus,A\otimes B},\\
 (\lambda_{\oplus,A}\otimes\lambda_{\oplus,B})
 \gamma_{0,0,A,B}(\delta\oplus1_{A\otimes B})
     &=\lambda_{\oplus,A\otimes B},\\
 \rho_{\otimes,A\oplus B}(1_{A\oplus B}\otimes\mu)
 \gamma_{A,1,B,1}
     &=\rho_{\otimes,A}\oplus\rho_{\otimes,B},\\
 \lambda_{\otimes,A\oplus B}(\mu\otimes1_{A\oplus B})
 \gamma_{1,A,1,B}
     &=\lambda_{\otimes,A}\oplus\lambda_{\otimes,B}.
\end{aligned}
\end{equation}
The remaining conditions in the usual definition say that
$(1,\mu,\tau)$ is a monoid for $\oplus$ and $(0,\delta,\tau)$
is a comonoid for $\otimes$. Here they hold automatically: the
monoid equations compare maps into the terminal object $1$, and
the comonoid equations compare maps out of the initial object $0$.

A duoidal structure is called \emph{normal} when its unit structure
maps are invertible; see \cite[Definition~2.6]{ShapiroSpivak2022}.
In the present setting this is equivalent to invertibility of
$\tau:0\to1$. Indeed, that condition makes both units zero objects,
and $\delta$ and $\mu$ are then isomorphisms between zero objects.
After identifying the units, the four equations
\eqref{duo:eq:units} reduce to
\begin{equation}\label{duo:eq:pointed-units}
\begin{aligned}
 \gamma_{A,B,0,0}&=1_{A\otimes B},&
 \gamma_{0,0,A,B}&=1_{A\otimes B},\\
 \gamma_{A,0,B,0}&=1_{A\oplus B},&
 \gamma_{0,A,0,B}&=1_{A\oplus B},
\end{aligned}
\end{equation}
using the canonical unit identifications. We do not use these
simplified equations before pointedness has been established.

For clarity, write $\iota_s^{X,Y}$ for the $s$th inclusion into
$X\oplus Y$ and $\pi_r^{X,Y}$ for the $r$th projection from
$X\otimes Y$. The following matrix uses only inclusions and
projections; no zero morphisms or identity comparison are required.

\begin{theorem}\label{duo:thm:matrix-interchange}
On a sum-product category with fixed sum and product structures, there
is at most one duoidal structure. Such a structure exists if and only
if, for every four objects $A,B,C,D$, the matrix
\begin{equation}\label{duo:eq:block-matrix}
 \begin{pmatrix}
 \iota_1^{A,C}\pi_1^{A,B}&\iota_2^{A,C}\pi_1^{C,D}\\
 \iota_1^{B,D}\pi_2^{A,B}&\iota_2^{B,D}\pi_2^{C,D}
 \end{pmatrix}
\end{equation}
has a realisation from
$S=(A\otimes B)\oplus(C\otimes D)$ to
$T=(A\oplus C)\otimes(B\oplus D)$. When it exists, this
realisation is $\gamma_{A,B,C,D}$. If these conditions hold, the following are equivalent:
\begin{enumerate}
\item The sum-product category is diagolinear.
\item The category is pointed.
\item The unit comparison $\tau:0\to1$ is invertible.
\item The duoidal structure is normal.
\end{enumerate}
In this case, after identifying the two units, the specialisation
$\gamma_{A,0,0,B}$ is the unique identity comparison
$i_{A,B}:A\oplus B\to A\otimes B$.
\end{theorem}

\begin{proof}
Let $j_1:A\otimes B\to S$ and $j_2:C\otimes D\to S$ be
the two source inclusions, and let $q_1:T\to A\oplus C$ and
$q_2:T\to B\oplus D$ be the two target projections. First suppose
that a duoidal structure is given. Its first unit axiom is the
commutativity of the square
\[
\xymatrix@C=3em@R=3em{
 (A\otimes B)\oplus0\ar[r]^{1\oplus\delta}
       \ar[d]_{\rho_{\oplus,A\otimes B}}
 &(A\otimes B)\oplus(0\otimes0)
       \ar[d]^{\gamma_{A,B,0,0}}\\
 A\otimes B
 &(A\oplus0)\otimes(B\oplus0)
       \ar[l]_{\rho_{\oplus,A}\otimes\rho_{\oplus,B}}.
}
\]
Let $u:0\to C$ and $v:0\to D$ be the unique maps. Initiality gives
$(u\otimes v)\delta={!}:0\to C\otimes D$. Hence
\[
 j_1=(1_{A\otimes B}\oplus(u\otimes v))
       (1_{A\otimes B}\oplus\delta)\rho_{\oplus,A\otimes B}^{-1}.
\]
Apply naturality of $\gamma$ at $1_A,1_B,u,v$ and then the displayed
unit square. This yields
\[
\begin{aligned}
 \gamma j_1
 &=((1_A\oplus u)\otimes(1_B\oplus v))
     \gamma_{A,B,0,0}(1_{A\otimes B}\oplus\delta)
       \rho_{\oplus,A\otimes B}^{-1}\\
 &=((1_A\oplus u)\rho_{\oplus,A}^{-1})\otimes
       ((1_B\oplus v)\rho_{\oplus,B}^{-1})
   =\iota_1^{A,C}\otimes\iota_1^{B,D}.
\end{aligned}
\]
Thus we obtain the commuting triangle
\[
\xymatrix@C=4em@R=3em{
 A\otimes B\ar[r]^{j_1}
       \ar[dr]_{\iota_1^{A,C}\otimes\iota_1^{B,D}}
 &S\ar[d]^{\gamma_{A,B,C,D}}\\
 &T.
}
\]
The left unit axiom gives the second source inclusion in the same
way. Thus
\begin{equation}\label{duo:eq:columns}
 \gamma j_1=\iota_1^{A,C}\otimes\iota_1^{B,D},
 \qquad
 \gamma j_2=\iota_2^{A,C}\otimes\iota_2^{B,D}.
\end{equation}
Dually, let $b:B\to1$ and $d:D\to1$ be the unique maps.
Terminality gives $\mu(b\oplus d)={!}:B\oplus D\to1$.
Naturality at $1_A,b,1_C,d$, followed by the third unit equation in
\eqref{duo:eq:units}, gives
\[
\begin{aligned}
 q_1\gamma
 &=\rho_{\otimes,A\oplus C}(1_{A\oplus C}\otimes\mu)
     \gamma_{A,1,C,1}((1_A\otimes b)\oplus(1_C\otimes d))\\
 &=(\rho_{\otimes,A}(1_A\otimes b))\oplus
       (\rho_{\otimes,C}(1_C\otimes d))
   =\pi_1^{A,B}\oplus\pi_1^{C,D}.
\end{aligned}
\]
The fourth unit equation gives the other projection, and therefore
\begin{equation}\label{duo:eq:rows}
 q_1\gamma=\pi_1^{A,B}\oplus\pi_1^{C,D},
 \qquad
 q_2\gamma=\pi_2^{A,B}\oplus\pi_2^{C,D}.
\end{equation}
Composing~\eqref{duo:eq:columns} with the two projections shows that
the matrix of $\gamma$ is precisely~\eqref{duo:eq:block-matrix}.
Matrix uniqueness proves uniqueness of the interchange. The three
unit structure maps are already forced by initiality or terminality.

Conversely, suppose that all the matrices
\eqref{duo:eq:block-matrix} have realisations, and denote their unique
realisations by $\gamma_{A,B,C,D}$. Equations
\eqref{duo:eq:columns} follow by composing each side with $q_1,q_2$:
the resulting composites agree by the specified matrix entries and
naturality of the projections. Joint monomorphicity gives the
equations themselves. Similarly, precomposition with $j_1,j_2$
and joint epimorphicity give~\eqref{duo:eq:rows}.

These equations also prove naturality. Indeed, for morphisms
$a:A\to A'$, $b:B\to B'$, $c:C\to C'$ and $d:D\to D'$,
the required equation is
\[
 ((a\oplus c)\otimes(b\oplus d))\gamma_{A,B,C,D}
 =\gamma_{A',B',C',D'}((a\otimes b)\oplus(c\otimes d)).
\]
After precomposition with $j_1$, both sides are
$(\iota_1^{A',C'}a)\otimes(\iota_1^{B',D'}b)$, by
\eqref{duo:eq:columns} and naturality of the sum inclusions.
After precomposition with $j_2$, both are
$(\iota_2^{A',C'}c)\otimes(\iota_2^{B',D'}d)$.
The source inclusions are jointly epimorphic, so the equation holds.

For the first diagram in~\eqref{duo:eq:associativity}, let
$e_r:A_r\to A_1\oplus A_2\oplus A_3$ and
$f_r:B_r\to B_1\oplus B_2\oplus B_3$ be the iterated sum
inclusions. Precompose each of the two paths $H\to K$ with the
$r$th inclusion $T_r\to H$. Applying~\eqref{duo:eq:columns}
at each step shows that both composites are $e_r\otimes f_r$.
This holds for $r=1,2,3$. The three inclusions into $H$ are
jointly epimorphic, and hence this diagram commutes.

For the second diagram, let
$p_r:A_1\otimes A_2\otimes A_3\to A_r$ and
$p'_r:B_1\otimes B_2\otimes B_3\to B_r$ be the iterated
product projections. Postcompose each of the two paths $V\to W$
with the $r$th projection $W\to U_r$. Applying
\eqref{duo:eq:rows} at each step shows that both composites are
$p_r\oplus p'_r$. Joint monomorphicity of the three projections
from $W$ proves commutativity. The canonical reassociations used here
preserve the relevant inclusions and projections, so these arguments
give the associativity conditions with all associators included.

It remains to verify the four unit equations, using the unique maps
$\delta,\mu,\tau$ of~\eqref{duo:eq:unit-maps}. For the first one,
\[
 (1_{A\otimes B}\oplus\delta)\rho_{\oplus,A\otimes B}^{-1}:
 A\otimes B\longrightarrow(A\otimes B)\oplus(0\otimes0)
\]
is precisely the first sum inclusion, since $\delta$ is the unique
map from $0$ to $0\otimes0$. Equation~\eqref{duo:eq:columns}
therefore sends it under $\gamma_{A,B,0,0}$ to
$\rho_{\oplus,A}^{-1}\otimes\rho_{\oplus,B}^{-1}$.
Postcomposing with $\rho_{\oplus,A}\otimes\rho_{\oplus,B}$ and
then composing on the right with $\rho_{\oplus,A\otimes B}$ gives
the first equation of~\eqref{duo:eq:units}. The second follows from
the second sum inclusion and the left unitors.

For the third equation, the composite
\[
 \rho_{\otimes,A\oplus B}(1_{A\oplus B}\otimes\mu):
 (A\oplus B)\otimes(1\oplus1)\longrightarrow A\oplus B
\]
is the first product projection, because $\mu$ is the unique map
from $1\oplus1$ to $1$. Applying~\eqref{duo:eq:rows} to
$\gamma_{A,1,B,1}$ identifies its composite with this projection as
$\rho_{\otimes,A}\oplus\rho_{\otimes,B}$, as required.
The second product projection and the left unitors give the fourth
equation. The monoid and comonoid equations for the unit structure
maps hold by terminality and initiality, respectively. Thus all
axioms of a duoidal structure have been verified.

Finally, a diagolinear category is pointed by definition. A category with an initial object $0$ and a
terminal object $1$ is pointed if and only if its unique map
$\tau:0\to1$ is invertible. As observed above, this is equivalent
to normality of the duoidal structure. Suppose these equivalent
conditions hold and identify the units with a zero object. Under
the canonical unit identifications, substituting $B=C=0$ in
\eqref{duo:eq:block-matrix}, with the fourth variable renamed $B$,
gives a morphism
\[
 i_{A,B}=\gamma_{A,0,0,B}:A\oplus B\longrightarrow A\otimes B
\]
whose matrix has diagonal entries $1_A,1_B$ and zero off-diagonal
entries. By Theorem~\ref{thm:identity-matrix} and
Lemma~\ref{compat from i}, these morphisms form an identity
comparison: they are natural and compatible with both unitors. The category is consequently diagolinear. Matrix uniqueness
also shows that this is its only identity comparison.
\end{proof}

\begin{theorem}\label{duo:thm:invertible-comparison-symmetry}
Let $\C$ be a diagolinear category in which every identity comparison
$i_{A,B}:A\oplus B\to A\otimes B$ is invertible. The following
conditions are equivalent:
\begin{enumerate}
\item The specified sum and product structures admit a duoidal structure.
\item The specified sum structure admits a symmetry.
\item The specified product structure admits a symmetry.
\end{enumerate}
When these conditions hold, both symmetries and the duoidal structure
are unique. The identity comparisons identify the two symmetric
monoidal structures, and every interchange $\gamma_{A,B,C,D}$ is
invertible.
\end{theorem}

\begin{proof}
First we check that the identity comparisons identify the two
monoidal structures coherently. Identify their units with a zero
object, as permitted by Lemma~\ref{prop:pointed}, and put
\[
\begin{aligned}
 I^L_{A,B,C}&=(i_{A,B}\otimes1_C)i_{A\oplus B,C},\\
 I^R_{A,B,C}&=(1_A\otimes i_{B,C})i_{A,B\oplus C}.
\end{aligned}
\]
The associator compatibility is commutativity of
\begin{equation}\label{duo:eq:comparison-associators}
\xymatrix@C=4em@R=3em{
 (A\oplus B)\oplus C\ar[r]^{\alpha_\oplus^{-1}}
       \ar[d]_{I^L_{A,B,C}}
 &A\oplus(B\oplus C)\ar[d]^{I^R_{A,B,C}}\\
 (A\otimes B)\otimes C\ar[r]_{\alpha_\otimes^{-1}}
 &A\otimes(B\otimes C).
}
\end{equation}
To verify it, let $e_1,e_2,e_3$ be the inclusions of $A,B,C$ into
the top-left object, and let $p_1,p_2,p_3$ be the projections from
the bottom-right object. For either composite $h$ around the square,
the binary identity-matrix equations give
\[
 p_rhe_s=
 \begin{cases}
  1_{X_r},&r=s,\\
  0,&r\ne s,
 \end{cases}
 \qquad (X_1=A,\ X_2=B,\ X_3=C).
\]
Indeed, each passage through a comparison keeps the selected
coordinate when its two indices agree and gives zero otherwise;
the associators preserve the ordered inclusions and projections.
The two composites consequently have the same triple identity
matrix and are equal by matrix uniqueness. The unit compatibilities
follow from Lemma~\ref{compat from i}. Since all the
comparisons are invertible, this proves the asserted coherent
identification of the monoidal structures. In particular, symmetries
transfer in either direction, with
\begin{equation}\label{duo:eq:transport-symmetry}
 \sigma^\otimes_{A,B}
   =i_{B,A}\sigma^\oplus_{A,B}i_{A,B}^{-1}.
\end{equation}
Thus (2) and (3) are equivalent.

Suppose (1) holds. Under the canonical unit identifications, the
specialisation
\[
 s_{A,B}=\gamma_{0,A,B,0}:A\oplus B\longrightarrow B\otimes A
\]
has matrix
\[
 \begin{pmatrix}0&1_B\\1_A&0\end{pmatrix}
\]
by~\eqref{duo:eq:block-matrix}. Comparing its two columns with the
columns of $i_{B,A}$ gives
\[
 s_{A,B}\iota_1^{A,B}=i_{B,A}\iota_2^{B,A},
 \qquad
 s_{A,B}\iota_2^{A,B}=i_{B,A}\iota_1^{B,A}.
\]
Here equality follows by composing with the two product projections,
which are jointly monomorphic. Hence the natural family
\[
 \sigma^\oplus_{A,B}=i_{B,A}^{-1}s_{A,B}:
 A\oplus B\longrightarrow B\oplus A
\]
interchanges the two sum inclusions. It follows that
$\sigma^\oplus_{B,A}\sigma^\oplus_{A,B}=1_{A\oplus B}$, since
this composite fixes both inclusions and they are jointly
epimorphic.

For completeness, the hexagon identities can also be checked on
inclusions. With canonical reassociations understood, they read
\[
\begin{aligned}
 \sigma^\oplus_{A,B\oplus C}
   &=(1_B\oplus\sigma^\oplus_{A,C})
       (\sigma^\oplus_{A,B}\oplus1_C),\\
 \sigma^\oplus_{A\oplus B,C}
   &=(\sigma^\oplus_{A,C}\oplus1_B)
       (1_A\oplus\sigma^\oplus_{B,C}).
\end{aligned}
\]
Both sides of the first equation send the inclusions of $A,B,C$ to
the corresponding inclusions in the order $B,C,A$; both sides of
the second do so in the order $C,A,B$. Joint epimorphicity of the
three source inclusions proves both equations. The unit equations
follow in the same way, since the inclusion of the remaining object
is a unitor isomorphism. Thus $\sigma^\oplus$ is a symmetry, proving (2).
Any symmetry must interchange the canonical inclusions, by its
naturality and unit equations. Joint epimorphicity therefore proves
uniqueness; equivalently, one may apply
Lemma~\ref{lem:symmetric-cover}. Formula~\eqref{duo:eq:transport-symmetry}
then gives uniqueness of the product symmetry as well.

Conversely, suppose (2) holds. Let
\[
 t_{A,B,C,D}:(A\oplus B)\oplus(C\oplus D)
       \longrightarrow(A\oplus C)\oplus(B\oplus D)
\]
be the isomorphism obtained by reassociating, applying
$1_A\oplus\sigma^\oplus_{B,C}\oplus1_D$, and reassociating again.
For a common symmetric tensor this is the usual duoidal
construction of~\cite[Example~3.2]{BookerStreet2013}. Define
\begin{equation}\label{duo:eq:invertible-interchange}
 \Gamma_{A,B,C,D}
 =i_{A\oplus C,B\oplus D}\,t_{A,B,C,D}
       (i_{A,B}^{-1}\oplus i_{C,D}^{-1}).
\end{equation}
This has the required source
$(A\otimes B)\oplus(C\otimes D)$ and target
$(A\oplus C)\otimes(B\oplus D)$.

Let $j_1,j_2$ be the inclusions into its source, and let
$\bar j_1,\bar j_2$ be the inclusions into
$(A\oplus B)\oplus(C\oplus D)$. Since $t$ interchanges the
middle two of the four ordered coordinates, testing the two
inclusions of each source summand gives
\[
\begin{aligned}
 t\bar j_1&=\iota_1^{A,C}\oplus\iota_1^{B,D},\\
 t\bar j_2&=\iota_2^{A,C}\oplus\iota_2^{B,D}.
\end{aligned}
\]
Naturality of $i$ now gives
\[
\begin{aligned}
 \Gamma j_1
 &=i_{A\oplus C,B\oplus D}
       (\iota_1^{A,C}\oplus\iota_1^{B,D})i_{A,B}^{-1}
   =\iota_1^{A,C}\otimes\iota_1^{B,D},\\
 \Gamma j_2
 &=i_{A\oplus C,B\oplus D}
       (\iota_2^{A,C}\oplus\iota_2^{B,D})i_{C,D}^{-1}
   =\iota_2^{A,C}\otimes\iota_2^{B,D}.
\end{aligned}
\]
Composing with the two target projections shows that $\Gamma$
realises~\eqref{duo:eq:block-matrix}. By
Theorem~\ref{duo:thm:matrix-interchange}, these realisations form the
unique duoidal structure. Every factor
in~\eqref{duo:eq:invertible-interchange} is invertible, so all of its
interchanges are invertible as well.
\end{proof}

\begin{lemma}\label{duo:lem:bicartesian-duoidal}
Every category with finite coproducts and finite products has a unique
duoidal structure on its specified coproduct and product structures.
Its interchange is the canonical map
\[
 (A\otimes B)\oplus(C\otimes D)
 \longrightarrow (A\oplus C)\otimes(B\oplus D)
\]
with matrix~\eqref{duo:eq:block-matrix}. Moreover, a functor strong for
both specified structures between duoidal sum-product categories
preserves their interchanges, using its monoidal comparison
isomorphisms to identify sources and targets.
\end{lemma}
\begin{proof}
The coproduct and product universal properties realise every matrix
of morphisms. In particular, the matrices
\eqref{duo:eq:block-matrix} have realisations, so existence and
uniqueness follow from Theorem~\ref{duo:thm:matrix-interchange}.
Explicitly, the interchange is the copairing
\[
 \bigl[\iota_1^{A,C}\otimes\iota_1^{B,D},
       \iota_2^{A,C}\otimes\iota_2^{B,D}\bigr].
\]
For the last assertion, a functor strong for both structures preserves
the canonical inclusions and projections under its comparison
isomorphisms. Applying it to the four matrix-entry equations for the
source interchange therefore gives the four equations defining the
target interchange. Matrix uniqueness makes the two morphisms equal.
The other duoidal structure maps are forced by initiality or
terminality and are preserved as well.
\end{proof}

\subsection{Duality and products}\label{prop:dual-products}

Let $\C$ be a sum-product category with sum structure $(\oplus,0)$
and product structure $(\otimes,1)$. Its opposite category
$\C^{\op}$ has sum structure $(\otimes^{\op},1)$ and product
structure $(\oplus^{\op},0)$. Indeed, initial and terminal objects
are interchanged by duality, as are joint epimorphicity and joint
monomorphicity. The canonical inclusions in $\C^{\op}$ are the
opposites of the canonical projections in $\C$, and conversely.

Consequently, duality transposes matrices and replaces their entries
by opposite morphisms. If
\[
 f:\bigoplus_{s=1}^{n}X_s\longrightarrow
   \bigotimes_{r=1}^{m}Y_r
\]
has matrix entries $f_{rs}=\pi_r f\iota_s$, then, using the
specified structures on $\C^{\op}$,
\[
 f^{\op}:\bigoplus_{r=1}^{m}Y_r\longrightarrow
         \bigotimes_{s=1}^{n}X_s
\]
has matrix entries
\[
 M_{\C^{\op}}(f^{\op})_{sr}
   =(\pi_r f\iota_s)^{\op}
   =f_{rs}^{\op}.
\]
Thus a matrix is realisable in $\C$ if and only if its transposed
matrix of opposite morphisms is realisable in $\C^{\op}$.
In particular, duality interchanges rows and columns.

If $\C$ is pointed, then $\C^{\op}$ is pointed, with the same zero
objects and with zero morphisms given by the opposites of those
in $\C$. Diagonal matrices therefore remain diagonal under
duality. It follows directly from the definition that $\C^{\op}$
is diagolinear whenever $\C$ is diagolinear. Their diagonal
realisations satisfy
\[
 D_{\C^{\op}}(f_1^{\op},\ldots,f_n^{\op})
   =D_{\C}(f_1,\ldots,f_n)^{\op}.
\]
In particular, the identity comparison in $\C^{\op}$ is
\[
 i^{\op}:\otimes^{\op}\Rightarrow\oplus^{\op},
 \qquad
 (i^{\op})_{A,B}=(i_{A,B})^{\op}.
\]
If $\C$ is bicartesian diagolinear, its specified product becomes
a coproduct in $\C^{\op}$, and its specified coproduct becomes a
product. Hence $\C^{\op}$ is also bicartesian diagolinear.
All three notions are therefore self-dual upon interchanging the
two specified structures.

Now let $(\C_k)_{k\in K}$ be a set-indexed family of sum-product
categories. Their product $\C=\prod_{k\in K}\C_k$ has sum and
product structures defined componentwise, including the units,
associators and unitors. The canonical inclusions and projections
are also defined componentwise. Their joint epimorphicity and
joint monomorphicity follow from the corresponding properties
in each factor, so $\C$ is a sum-product category.

A matrix in $\C$ is equivalently a family of matrices of the same
dimensions, one in each factor. Writing $f=(f^{(k)})_{k\in K}$,
we have
\[
 M_{\C}(f)_{rs}
   =\bigl(M_{\C_k}(f^{(k)})_{rs}\bigr)_{k\in K}.
\]
A matrix in $\C$ is therefore realisable if and only if each
component matrix is realisable; its realisation is then the tuple
of the unique component realisations.

If every $\C_k$ is pointed, then $\C$ is pointed, with zero objects
and zero morphisms given componentwise. Diagonal and identity
matrices are consequently determined componentwise. Hence products
of diagolinear categories are diagolinear, with
\[
 D_{\C}(f_1,\ldots,f_n)
   =\bigl(
      D_{\C_k}(f_1^{(k)},\ldots,f_n^{(k)})
     \bigr)_{k\in K}.
\]
In particular, their identity comparisons are given by
\[
 i^{\C}_{A,B}
   =\bigl(i^{\C_k}_{A^{(k)},B^{(k)}}\bigr)_{k\in K}.
\]
Finally, the universal properties of coproducts and products hold
componentwise. Thus products of bicartesian diagolinear categories
are bicartesian diagolinear. These statements include the empty
product, which is the terminal category with its unique sum and
product structures.

\section{Bicartesian reflection}
\label{sec:bicartesian-reflection}\label{sec:linearisation}

\subsection{Bicartesian diagolinear categories}

\begin{definition}\label{bic:def:bicartesian}
A diagolinear category is \emph{bicartesian} if its specified sum
structure is a coproduct structure and its specified product structure
is a product structure. Thus it is a pointed category with specified
finite coproducts and finite products. No equalisers or coequalisers
are required.
\end{definition}

Thus a bicartesian diagolinear category is precisely a pointed
bicartesian sum-product category.

\begin{samepage}
\begin{theorem}\label{bic:prop:all-matrices}
For a diagolinear category, the following are equivalent:
\begin{enumerate}
\item It is bicartesian.
\item Every finite row and every finite column of morphisms has a
realisation.
\item Every finite matrix of morphisms has a realisation from its
specified source sum to its specified target product.
\end{enumerate}
\end{theorem}
\end{samepage}
\begin{proof}
The universal properties of coproducts and products give the second
assertion from the first. Conversely, realisability of every row is
exactly the missing existence assertion in the coproduct universal
property: uniqueness is already provided by joint epimorphicity of
the inclusions. Dually, realisability of every column makes the
specified products genuine products. The units supply the nullary
universal properties. This proves the equivalence of the first two
assertions. The second implies the third by
Lemma~\ref{diag:prop:row-column}. Finally, one-row and one-column
matrices give the second assertion from the third.
\end{proof}

Bicartesianness concerns existence of realisations. It does not require
the identity comparison to be invertible. If this comparison is also
invertible, the coproducts and products are finite biproducts; such a
category is called \emph{linear}. Its hom-sets
have canonical commutative-monoid structures and composition is
bilinear. It is \emph{additive} when these monoids are abelian groups.
The reflection studied below has bicartesian diagolinear categories
as its target; linear categories occur as a special case.

\subsection{The $2$-category of diagolinear categories}
\label{refl:sec:2categories}

The universal constructions in this paper preserve the specified
structures up to coherent isomorphism. They do not require a functor to
preserve a chosen tensor object, unit, associator or unitor on the nose.
We make the relevant 2-categories explicit here. The reflection into
bicartesian diagolinear categories is constructed below.
For size purposes we fix a universe, take the structured categories to
be small in it, and form the 2-categories in a larger universe.

Let $F:\C\to\D$ be a functor between diagolinear categories that is
strong monoidal for both the sum and the product structures. Thus it has
natural isomorphisms
\[
\begin{aligned}
\phi^{F,\oplus}_{X,Y}&:FX\oplus FY\longrightarrow F(X\oplus Y),
&\phi^{F,\oplus}_0&:0_{\D}\longrightarrow F0_{\C},\\
\phi^{F,\otimes}_{X,Y}&:FX\otimes FY\longrightarrow F(X\otimes Y),
&\phi^{F,\otimes}_0&:1_{\D}\longrightarrow F1_{\C},
\end{aligned}
\]
satisfying the usual associativity and unit axioms for a strong
monoidal functor. Compatibility with the identity comparisons means
commutativity of the square
\begin{equation}\label{refl:eq:strong-comparison}
\vcenter{\hbox{$\xymatrix@C=4em@R=3em{
 FX\oplus FY\ar[r]^{\phi^{F,\oplus}_{X,Y}}
       \ar[d]_{i^{\D}_{FX,FY}}
 &F(X\oplus Y)\ar[d]^{F(i^{\C}_{X,Y})}\\
 FX\otimes FY\ar[r]_{\phi^{F,\otimes}_{X,Y}}
 &F(X\otimes Y).
}$}}
\end{equation}

For functors $F,G:\C\to\D$ strong monoidal for a monoidal
structure $(\star,I)$, a natural transformation $\alpha:F\Rightarrow G$
is \emph{monoidal} when it commutes with the binary and unit
comparison isomorphisms, that is,
\[
\begin{aligned}
\alpha_{X\star Y}\phi^{F,\star}_{X,Y}
 &=\phi^{G,\star}_{X,Y}(\alpha_X\star\alpha_Y),\\
\alpha_{I_{\C}}\phi^{F,\star}_0&=\phi^{G,\star}_0.
\end{aligned}
\]
For the sum structure take $(\star,I)=(\oplus,0)$; for the product
structure take $(\star,I)=(\otimes,1)$.

\begin{lemma}\label{refl:lem:strong-prelinear}
For diagolinear categories $\C$ and $\D$, the following assertions hold.
\begin{enumerate}
\item A functor $F:\C\to\D$ strong monoidal for both the sum and
product structures preserves zero morphisms and
satisfies~\eqref{refl:eq:strong-comparison} automatically.
\item For a fixed functor $F:\C\to\D$, any two choices of the
binary and unit comparison isomorphisms making $F$ strong monoidal
for both structures coincide. Thus these isomorphisms, when they
exist, are determined by the functor $F$.
\item Every natural transformation between functors $\C\to\D$
that are strong monoidal for both structures is monoidal for both
structures.
\end{enumerate}
\end{lemma}
\begin{proof}
The isomorphism $0_{\D}\to F0_{\C}$ makes $F0_{\C}$ a zero object.
Applying $F$ to a factorisation through $0_{\C}$ therefore proves
preservation of zero morphisms. Naturality of the monoidal comparisons
at arrows to and from the units, followed by the monoidal unit axioms,
gives
\begin{equation}\label{refl:eq:strong-axes}
\phi^{F,\oplus}_{X,Y}\iota_s=F\iota_s,
\qquad
F\pi_r\,\phi^{F,\otimes}_{X,Y}=\pi_r
\qquad(r,s=1,2).
\end{equation}
Here the inclusions are those of the sum structures and the projections
are those of the product structures. For example, the first formula
for $s=1$ is obtained by applying naturality to
$1_X\oplus(0_{\C}\to Y)$ and then the right-unit axiom; the formula
for $s=2$ uses the left-unit axiom. Reversing the arrows gives the two
projection formulas.

Consider
\[
k=(\phi^{F,\otimes}_{X,Y})^{-1}
       F(i^{\C}_{X,Y})\phi^{F,\oplus}_{X,Y}
       :FX\oplus FY\longrightarrow FX\otimes FY.
\]
By~\eqref{refl:eq:strong-axes} and Theorem~\ref{thm:identity-matrix},
\[
\pi_r k\iota_s
 =F(\pi_r i^{\C}_{X,Y}\iota_s)
 =\begin{cases}1_{FX},&r=s=1,\\
                 1_{FY},&r=s=2,\\
                 0,&r\ne s.
   \end{cases}
\]
The same is the matrix of $i^{\D}_{FX,FY}$. Matrix uniqueness in
$\D$ gives $k=i^{\D}_{FX,FY}$, proving the first assertion.

The first equation in~\eqref{refl:eq:strong-axes} determines
$\phi^{F,\oplus}_{X,Y}$, since its source inclusions are jointly
epimorphic. The second equation can be rewritten
\[
\pi_r(\phi^{F,\otimes}_{X,Y})^{-1}=F\pi_r.
\]
It determines the inverse product comparison, since the projections
in $\D$ are jointly monomorphic. The unit comparisons are unique maps
between zero objects. This proves the second assertion.

Let $\alpha:F\Rightarrow G$ be any natural transformation. For the
sum structure the required square is
\[
\xymatrix@C=4em@R=3em{
 FX\oplus FY\ar[r]^{\phi^{F,\oplus}_{X,Y}}
             \ar[d]_{\alpha_X\oplus\alpha_Y}
 &F(X\oplus Y)\ar[d]^{\alpha_{X\oplus Y}}\\
 GX\oplus GY\ar[r]_{\phi^{G,\oplus}_{X,Y}}
 &G(X\oplus Y).
}
\]
After precomposition with the $s$th inclusion, its two composites are
$\alpha_{X\oplus Y}F\iota_s$ and $G\iota_s\alpha_{X_s}$,
where $X_1=X$ and $X_2=Y$. They agree by naturality of $\alpha$.
The inclusions are jointly epimorphic, so the square commutes.
For the analogous product square, postcompose with $G\pi_r$.
Its composites agree by naturality of $\alpha$ and the projection
equations in~\eqref{refl:eq:strong-axes}. The family $G\pi_1,G\pi_2$
is jointly monomorphic: after composition with the isomorphism
$\phi^{G,\otimes}_{X,Y}$ it is the canonical projection family in
$\D$. Thus the product square also commutes. Finally, the two unit
compatibilities hold by uniqueness of maps between zero objects.
\end{proof}

\begin{definition}\label{refl:def:strong-functor}
A \emph{structure-preserving functor} $F:\C\to\D$ between diagolinear
categories is strong monoidal for both specified structures.
By Lemma~\ref{refl:lem:strong-prelinear}, compatibility with the
identity comparisons is automatic. Compatibility of the unit
comparisons with $j:0\to1$ is automatic since all the objects concerned
are zero objects.
\end{definition}

\begin{definition}\label{refl:def:2categories}
Let $\DiagLin$ be the 2-category whose objects are diagolinear
categories, whose 1-cells are the structure-preserving functors of
Definition~\ref{refl:def:strong-functor}, and whose 2-cells are ordinary
natural transformations. Let $\BiCart$ be its full sub-2-category
whose objects are the bicartesian diagolinear categories of
Definition~\ref{bic:def:bicartesian}.
\end{definition}

The subscript $0$ records pointedness. The 1-cells in $\BiCart$ are
precisely the functors preserving finite coproducts and finite products
up to their canonical isomorphisms. Indeed, their strong comparisons
carry the specified universal diagrams to universal diagrams;
conversely, preservation of those diagrams supplies the unique strong
comparison maps. Their associativity and unit axioms follow by testing
canonical inclusions and projections. The 2-cells are all natural
transformations, by Lemma~\ref{refl:lem:strong-prelinear}.

Composition of strong monoidal functors uses the usual composite
comparison isomorphisms. Together with ordinary composition of natural
transformations, it makes $\DiagLin$ and $\BiCart$ 2-categories.
The comparison maps, when they exist, are uniquely determined by the
underlying functor. They need not be identities.

A \emph{biadjunction} between these 2-categories is expressed by
pseudonatural equivalences of hom-categories. When a chosen construction
gives 2-natural isomorphisms of hom-categories, it gives a
\emph{2-adjunction}. Both conventions allow strong monoidal 1-cells;
they do not require strict preservation of the specified tensor
objects or units. We use the same symbols for the underlying ordinary
categories when discussing hom-set adjunctions. Background on these
2-dimensional universal properties can be found in~\cite{Lack2010}.

\subsection{The reflection} 
Let $\C$ be diagolinear. We write $[f]:X\to Y$ for a formal copy of
an arrow $f:X\to Y$ of $\C$. Keep the objects of $\C$, and form
well-typed arrow terms by the following recursive rules:
\begin{enumerate}
\item Every $[f]$ is a term, including the original identities and
zero arrows.
\item If $u:X\to Y$ and $v:Y\to Z$ are terms, then $vu:X\to Z$
is a term.
\item If $u:Z\to X$ and $v:Z\to Y$ are terms, then
$\langle u,v\rangle:Z\to X\otimes Y$ is a term.
\item If $u:X\to Z$ and $v:Y\to Z$ are terms, then
$[u,v]:X\oplus Y\to Z$ is a term.
\end{enumerate}
The brackets around one original arrow and the brackets for a
copairing of two terms have distinct meanings. In particular, the
pairing and copairing rules apply to terms already obtained by any of
these rules, not only to original arrows.

Take the least equivalence relation on each set of parallel terms
which is stable under composition, pairing and copairing, and contains
the following equations. Impose the category axioms with identities
$[1_X]$, and the equations
\[
 [g][f]=[gf]
\]
for original composable arrows. For every term $u:X\to Y$, impose
zero absorption with the appropriate original zero arrows:
\begin{equation}\label{bic:eq:zero-absorption}
 u[0_{W,X}]=[0_{W,Y}],\qquad
 [0_{Y,Z}]u=[0_{X,Z}].
\end{equation}
Finally impose the product and coproduct equations, for all well-typed
terms:
\begin{equation}\label{bic:eq:pairing-equations}
\begin{aligned}
 [\pi_1]\langle u,v\rangle&=u,&
 [\pi_2]\langle u,v\rangle&=v,&
 \langle[\pi_1]w,[\pi_2]w\rangle&=w,\\
 [u,v][\iota_1]&=u,&
 [u,v][\iota_2]&=v,&
 [w[\iota_1],w[\iota_2]]&=w.
\end{aligned}
\end{equation}
Here $\iota_r$ and $\pi_r$ are the original canonical sum inclusions
and product projections for the displayed pair of objects. The last
column applies whenever $w$ has the required product codomain or
coproduct domain. A least such congruence exists: the indiscrete
relations on parallel terms satisfy the requirements, and
intersections preserve all of them.

\begin{definition}\label{bic:def:reflection}
The category $\Bic(\C)$ has the objects of $\C$ and the equivalence
classes of the above arrow terms as its morphisms. Composition is
induced by composition of terms. Its canonical functor is
\[
 b_{\C}:\C\longrightarrow\Bic(\C),\qquad
 X\longmapsto X,\quad f\longmapsto[f].
\]
We use the same notation for a term and its class when no ambiguity
arises.
\end{definition}

\begin{theorem}\label{bic:thm:bicartesian}
The category $\Bic(\C)$ is bicartesian diagolinear. Its specified
coproduct of $X,Y$ is the original object $X\oplus Y$, with
inclusions $[\iota_1],[\iota_2]$, and its specified product is the
original object $X\otimes Y$, with projections $[\pi_1],[\pi_2]$.
Both original monoidal structures extend to this category, with their
original tensor objects and the images of their structural
isomorphisms. The functor $b_{\C}$ is strict for these chosen
structures, and the identity comparison of $\Bic(\C)$ is
$[i_{X,Y}]$.
\end{theorem}
\begin{proof}
The category axioms and compatibility of the congruence with
composition make $\Bic(\C)$ a category, and the original composition
equations make $b_{\C}$ a functor. The original zero objects remain
zero objects. Indeed, $[1_0]=[0_{0,0}]$, so any term $u:0\to X$
satisfies
\[
 u=u[1_0]=u[0_{0,0}]=[0_{0,X}]
\]
by~\eqref{bic:eq:zero-absorption}. Similarly any term $v:X\to0$
equals $[0_{X,0}]$. The same argument applies to the original product
unit $1$, which is also a zero object. Thus the two units retain their
initial and terminal properties.

Given $u:Z\to X$ and $v:Z\to Y$ in the quotient, take any
representative terms and form $\langle u,v\rangle$. Stability of the
congruence under pairing makes its class independent of the
representatives. The first two equations of the first row
of~\eqref{bic:eq:pairing-equations} give the dotted arrow in
\[
\xymatrix@C=3em@R=3em{
 &Z\ar[dl]_u\ar@{.>}[d]^{\langle u,v\rangle}\ar[dr]^v\\
 X&X\otimes Y\ar[l]^{[\pi_1]}\ar[r]_{[\pi_2]}&Y.
}
\]
If $w$ has the same two components, the third equation gives
$w=\langle[\pi_1]w,[\pi_2]w\rangle=\langle u,v\rangle$.
This proves the product universal property. The second row gives the
coproduct universal property in exactly the same way:
\[
\xymatrix@C=3em@R=3em{
 X\ar[r]^{[\iota_1]}\ar[dr]_u
 &X\oplus Y\ar@{.>}[d]^{[u,v]}
 &Y\ar[l]_{[\iota_2]}\ar[dl]^v\\
 &Z.
}
\]
Binary products and coproducts, together with their units, give all
finite products and coproducts.

For arrows $u:X\to X'$ and $v:Y\to Y'$ in the quotient, define
\begin{equation}\label{bic:eq:extended-tensors}
 u\oplus v=[[\iota'_1]u,[\iota'_2]v],\qquad
 u\otimes v=\langle u[\pi_1],v[\pi_2]\rangle.
\end{equation}
The universal properties show that these operations preserve
identities and composition. On original arrows they agree with the
images of the original tensors. For example,
\[
 [f\oplus g][\iota_r]=[\iota'_r][f_r]
 \qquad(f_1=f,\ f_2=g)
\]
is an original naturality equation, and uniquely determines the
coproduct arrow $[f]\oplus[g]$. Dually the projection equations give
$[f\otimes g]=[f]\otimes[g]$.

The canonical coproduct associator is characterised by preserving the
three ordered inclusions. The original sum associator already
satisfies those equations, so its image is this canonical associator.
The same argument applies to the sum unitors, and its dual applies to
the product associators and unitors. Consequently their images are
natural for all the new arrows and satisfy the monoidal coherence
axioms. This proves strict preservation by $b_{\C}$ in the chosen
construction.

Finally, put $X_1=X$ and $X_2=Y$. The original identity-matrix equations imply
\[
 [\pi_r][i_{X,Y}][\iota_s]
 =\begin{cases}[1_{X_s}],&r=s,\\{[0_{X_s,X_r}]},&r\ne s.
 \end{cases}
\]
The coproduct and product universal properties therefore identify
$[i_{X,Y}]$ with the canonical identity comparison. That comparison is
natural for every pair of arrows in the new category and is compatible
with the units. Thus $\Bic(\C)$ is diagolinear as asserted.
\end{proof}

In particular, for arbitrary arrows $a:X_1\to Y_1$,
$b:X_2\to Y_1$, $c:X_1\to Y_2$, $d:X_2\to Y_2$ of the new
category, the two terms
\begin{equation}\label{bic:eq:matrix-realisation}
 [\langle a,c\rangle,\langle b,d\rangle]
 =\langle[a,b],[c,d]\rangle
 :X_1\oplus X_2\longrightarrow Y_1\otimes Y_2
\end{equation}
are equal: both have matrix $\left[\begin{smallmatrix}a&b\\c&d\end{smallmatrix}\right]$.
Iterating gives every finite matrix realisation. If a matrix of old
arrows already had a realisation in $\C$, its image under $b_{\C}$
is this same realisation, by matrix uniqueness. Thus the construction
preserves all old realisations, although distinct old arrows may
acquire the same image.

\begin{theorem}\label{bic:thm:universal-extension}
Let $\C$ be diagolinear and $\D$ bicartesian diagolinear. Every
strong structure-preserving functor $F:\C\to\D$ has a unique
strong structure-preserving extension
\[
\xymatrix@C=4em@R=3em{
 \C\ar[r]^{b_{\C}}\ar[dr]_F
 &\Bic(\C)\ar@{.>}[d]^{\overline F}\\
 &\D.
}
\]
The extension has the same object map and the same comparison
isomorphisms as $F$. It preserves finite products and finite
coproducts.
\end{theorem}
\begin{proof}
Interpret each original term $[f]$ as $Ff$, and interpret a composite
as the composite of its interpretations. On pairing and copairing
terms use
\begin{equation}\label{bic:eq:extension-terms}
\begin{aligned}
 \overline F\langle u,v\rangle
   &=\phi^{F,\otimes}_{X,Y}
        \langle\overline Fu,\overline Fv\rangle,
       &&u:Z\to X,\ v:Z\to Y,\\
 \overline F[u,v]
   &=[\overline Fu,\overline Fv]
        (\phi^{F,\oplus}_{X,Y})^{-1},
       &&u:X\to Z,\ v:Y\to Z.
\end{aligned}
\end{equation}
The pairings and copairings on the right are those of $\D$. The
comparison maps are essential: $F(X\otimes Y)$ and $F(X\oplus Y)$
need not equal $FX\otimes FY$ and $FX\oplus FY$.

The original equations and category axioms are respected because $F$
is a functor. Zero absorption is respected since $F$ preserves zero
arrows by Lemma~\ref{refl:lem:strong-prelinear}. For the product
equations, the canonical-axis identities give
\[
 F\pi_r\,\phi^{F,\otimes}_{X,Y}=\pi_r.
\]
Postcomposing the first formula of~\eqref{bic:eq:extension-terms}
with $F\pi_1,F\pi_2$ therefore gives $\overline Fu,\overline Fv$.
Moreover these two maps are jointly monomorphic, because after
composition with the isomorphism $\phi^{F,\otimes}_{X,Y}$ they are
the product projections in $\D$. This proves the interpretation of
the product uniqueness equation for every term $w$. Dually,
\[
 \phi^{F,\oplus}_{X,Y}\iota_r=F\iota_r
\]
proves the coproduct equations. Consequently the interpretation respects
every generating equation and the congruence they generate, so it
defines a functor $\overline F:\Bic(\C)\to\D$.

The original comparison maps make $\overline F$ strong monoidal for
both extended tensors. We check their naturality for the new arrows.
For $u:X\to X'$ and $v:Y\to Y'$, the two maps in the proposed
sum-comparison square
\[
\xymatrix@C=4em@R=3em{
 FX\oplus FY\ar[r]^{\phi^{F,\oplus}_{X,Y}}
      \ar[d]_{\overline Fu\oplus\overline Fv}
 &F(X\oplus Y)\ar[d]^{\overline F(u\oplus v)}\\
 FX'\oplus FY'\ar[r]_{\phi^{F,\oplus}_{X',Y'}}
 &F(X'\oplus Y')
}
\]
agree after precomposition with each inclusion. For instance, the first
restriction of either composite is $F\iota'_1\,\overline Fu$,
using~\eqref{bic:eq:extended-tensors} and the inclusion equations
above. Joint epimorphicity proves the square. The product square is
verified dually by postcomposing with the two maps $F\pi'_r$, which
are jointly monomorphic. The associativity and unit coherence
equations involve the same object-indexed comparison maps and images
of the same original structural arrows as those of $F$. Hence they
also hold. Since the specified tensors in the source and target are
actual products and coproducts, these strong comparisons express
preservation of the corresponding universal diagrams.

For uniqueness, any extension has the same object map. Its comparisons
are the comparisons of $F$, because $b_{\C}$ is strict and leaves
objects unchanged; equivalently, their uniqueness follows from
Lemma~\ref{refl:lem:strong-prelinear}. Preservation of the product and
coproduct diagrams then forces exactly the two formulas
in~\eqref{bic:eq:extension-terms}. Induction on the formation of terms
forces the value of the extension on every morphism.
\end{proof}

\begin{lemma}\label{bic:prop:functor-categories}
For every diagolinear category $\C$ and bicartesian diagolinear
category $\D$, restriction along $b_{\C}$ gives an isomorphism of
hom-categories
\begin{equation}\label{bic:eq:hom-isomorphism}
 \BiCart(\Bic(\C),\D)\xrightarrow{\ \cong\ }
 \DiagLin(\C,U\D),
\end{equation}
where $U:\BiCart\to\DiagLin$ is the full inclusion. Morphisms in
these hom-categories are ordinary natural transformations.
\end{lemma}
\begin{proof}
The preceding theorem gives the bijection on functors. Let
$\alpha:F\Rightarrow G$ be a natural transformation. Since $b_{\C}$
is the identity on objects, an extension of $\alpha$ can have only the
components $\alpha_X$. We verify that these components are natural
on every arrow term.

Naturality holds on original arrows by hypothesis and is preserved by
composition. Suppose it holds on $u:Z\to X$ and $v:Z\to Y$.
For the proposed equation on $w=\langle u,v\rangle$, postcompose
both sides with $G\pi_r$:
\[
\begin{aligned}
 G\pi_r\,\alpha_{X\otimes Y}\,\overline Fw
 &=\alpha_{X_r}F\pi_r\,\overline Fw
  =\alpha_{X_r}\overline Fu_r,\\
 G\pi_r\,\overline Gw\,\alpha_Z
 &=\overline Gu_r\,\alpha_Z,
\end{aligned}
\]
where $X_1=X$, $X_2=Y$, $u_1=u$, $u_2=v$. The two results agree by the inductive
hypothesis. The maps $G\pi_1,G\pi_2$ are jointly monomorphic by the
strong product comparison for $G$, proving naturality on $w$.
For a copairing $w=[u,v]:X\oplus Y\to Z$, precompose the proposed
equation with $F\iota_r$ instead. The resulting equations are
$\alpha_Z\overline Fu_r=\overline Gu_r\alpha_{X_r}$, and the maps
$F\iota_1,F\iota_2$ are jointly epimorphic by the strong coproduct
comparison for $F$. This proves naturality by induction.

Extension and restriction leave components unchanged, so they are
inverse and preserve identity transformations and vertical composition.
Every such transformation is automatically monoidal for both
structures by Lemma~\ref{refl:lem:strong-prelinear}. Thus they give
the asserted isomorphism of categories.
\end{proof}

\begin{theorem}\label{bic:thm:reflection}
The construction extends to a 2-functor
\[
 \Bic:\DiagLin\longrightarrow\BiCart
\]
left 2-adjoint to the full inclusion $U$. Its unit is $b$, and its
counit at $\D\in\BiCart$ evaluates the formal terms in $\D$:
\[
 \varepsilon_{\D}:\Bic(U\D)\longrightarrow\D.
\]
The counit is an isomorphism. Consequently $\BiCart$ is reflective
in $\DiagLin$, also on underlying ordinary categories, and
$\Bic(\Bic(\C))\cong\Bic(\C)$. With equivalent choices of the
specified structures and their presentations, the same universal
property is expressed as a biadjunction.
\end{theorem}
\begin{proof}
For a strong functor $F:\C\to\D$, apply
Theorem~\ref{bic:thm:universal-extension} to $b_{\D}F$. This defines
$\Bic(F)$ and the commutative square
\[
\xymatrix@C=4em@R=3em{
 \C\ar[r]^{b_{\C}}\ar[d]_F
 &\Bic(\C)\ar[d]^{\Bic(F)}\\
 \D\ar[r]_{b_{\D}}&\Bic(\D).
}
\]
In particular, $\Bic(F)([f])=[Ff]$. Its values on arbitrary terms
are given by~\eqref{bic:eq:extension-terms}, with the comparisons
$[\phi^{F,\oplus}]$ and $[\phi^{F,\otimes}]$. Uniqueness of
extension proves preservation of identity functors and composition.
For a natural transformation $\alpha:F\Rightarrow G$, extend
$b_{\D}\alpha$ by Lemma~\ref{bic:prop:functor-categories}.
Thus
\[
 \Bic(\alpha)_X=[\alpha_X].
\]
Functoriality of $b_{\D}$ proves preservation of vertical composition;
the equality $\Bic(F)([f])=[Ff]$ proves preservation of whiskering and
hence of horizontal composition. This gives a 2-functor.

Restriction and extension in~\eqref{bic:eq:hom-isomorphism} are
2-natural in both variables. In the source variable this follows from
the displayed naturality square for $b$ and uniqueness of extension.
In the target variable it follows because a product-and-coproduct
preserving functor takes an interpreted term to its interpretation
after composition. These assertions apply to transformations as well,
since the extension leaves their components unchanged. The resulting
2-natural isomorphisms of hom-categories prove the 2-adjunction. On
objects of the hom-categories they give the underlying ordinary
adjunction.

The counit is the unique extension of the identity of $\D$. Since
$\D$ already has the specified products and coproducts,
$\varepsilon_{\D}$ sends $[f]$ to $f$ and every formal pairing or
copairing to its actual pairing or copairing in $\D$. Clearly
$\varepsilon_{\D}b_{U\D}=1_{\D}$. Conversely, $b_{U\D}$ preserves
actual pairings and copairings: their images have the prescribed
projections and inclusions, so the universal properties identify them
with the pairings and copairings in $\Bic(U\D)$. Induction on terms
now gives
\[
 b_{U\D}\varepsilon_{\D}=1_{\Bic(U\D)}.
\]
Thus the counit is an isomorphism, with inverse $b_{U\D}$.
The other triangle identity is equally explicit: the composite
$\varepsilon_{\Bic(\C)}\Bic(b_{\C})$ fixes each original generator
and preserves pairings, copairings and composition, so fixes every
term. The unit and counit are 2-natural by the same formulas on
functors and transformation components. This completes the claimed
reflection and its 2-categorical description.
\end{proof}

\subsection{Properties of the reflection}

\begin{theorem}[Invertible comparisons and a presentation of the reflection]
\label{bic:prop:invertible-comparison}
\label{bic:lem:invertible-comparison}
\label{bic:prop:common-tensor-presentation}
\label{lin:lem:homwise-congruence}
Let $\C$ be a diagolinear category whose identity comparison $i$ is
invertible. Then $\Bic(\C)$ is linear: its specified products and
coproducts are biproducts. Moreover, every morphism of $\Bic(\C)$
is a finite sum of morphisms $[f]=b_{\C}(f)$ with $f$ an original
morphism of $\C$. Put
\[
 p_r=\pi_r i_{X,Y}\qquad(r=1,2),
\]
so that $p_1:X\oplus Y\to X$ and $p_2:X\oplus Y\to Y$.
The hom-monoids admit the presentation
\[
 \Bic(\C)(A,B)\cong\N[\C(A,B)]/\!\sim_{A,B},
\]
where $\N[\C(A,B)]$ is the free commutative monoid on $\C(A,B)$,
and $\sim_{A,B}$ is the commutative-monoid congruence generated by
$[0]=0$ and
\begin{equation}\label{lin:eq:split-relation}
 [vu]=[v\iota_1p_1u]+[v\iota_2p_2u],
 \qquad A\xrightarrow{u}X\oplus Y\xrightarrow{v}B.
\end{equation}
Here $u,v$ are original arrows. Composition is the bilinear
extension of $[g][f]=[gf]$, and the unit sends $f$ to $[f]$.
In particular, these are all the relations needed for this
presentation.
\end{theorem}

\begin{proof}
The identity comparison in $\Bic(\C)$ is $[i_{X,Y}]$ by
Theorem~\ref{bic:thm:bicartesian}. Under the hypothesis its inverse
is $[i_{X,Y}^{-1}]$. Every object of $\Bic(\C)$ is an original
object, so this holds for every pair of objects of the new category.
It follows that the coproduct and product of every pair are
canonically isomorphic through their identity comparison. Each is
therefore a biproduct.

We use the resulting commutative-monoid addition on hom-sets. In
terms of the chosen product and coproduct, it is
\[
 a+b=[a,b]\,[i_{X,X}]^{-1}\langle1_X,1_X\rangle
 \qquad(a,b:X\to Y).
\]
The biproduct identities give its associativity, commutativity and
zero; the same identities show that composition distributes over
this addition. For the product object $X\otimes Y$, its biproduct
inclusions are the old arrows $[i_{X,Y}\iota_1]$ and
$[i_{X,Y}\iota_2]$. For the coproduct object $X\oplus Y$, its
biproduct projections are $[\pi_1i_{X,Y}]$ and
$[\pi_2i_{X,Y}]$. Hence
\begin{equation}\label{bic:eq:terms-as-sums}
\begin{aligned}
 \langle u,v\rangle
   &=[i\iota_1]u+[i\iota_2]v,\\
 [u,v]&=u[\pi_1i]+v[\pi_2i].
\end{aligned}
\end{equation}
These equations follow by comparing projections or restrictions to
inclusions, respectively.

Induct on the construction of an arrow term. An original generator
is already a one-term sum. A composite of two finite sums is a
finite sum of original composites, by bilinearity and
$[g][f]=[gf]$. For pairings and copairings,
\eqref{bic:eq:terms-as-sums} reduces the claim to the inductive
hypotheses, since its coefficients are images of original arrows.
Every term is therefore represented by a finite sum of original
arrows. The empty sum is the zero morphism.

For the presentation, use the normalisation of
Theorem~\ref{diag:thm:normalisation}: the common monoidal operation is
$\oplus$, with inclusions $\iota_r$ and projections
$p_r=\pi_r i$. Form a category $\mathsf Q$ on the original objects
with hom-monoids
\[
 \mathsf Q(A,B)=\N[\C(A,B)]/\!\sim_{A,B},
\]
using precisely the relations in the statement.

To see that bilinear composition is well defined, first consider
composition with an original arrow. Precomposing a generating
relation~\eqref{lin:eq:split-relation} replaces $u$ by its composite
with that arrow; postcomposing replaces $v$ by the corresponding
composite. Both give another generating relation. Since these
operations preserve addition, they preserve the generated
congruences. Composition with a finite sum is a sum of such
operations, so it also preserves the congruences. The equations
$[0]=0$ respect composition because original zero arrows compose
to zero. Associativity and identity laws are inherited from $\C$.
Its zero objects remain zero objects: each hom-set to or from one
originally consists of a single zero arrow, whose generator is
identified with the empty sum. We obtain a functor
\[
 q:\C\longrightarrow\mathsf Q,\qquad f\longmapsto[f].
\]

The coordinate equations
\[
 p_r\iota_s=\delta_{rs}
\]
hold in $\C$, where $\delta_{rs}$ denotes the appropriate identity
or zero arrow. Taking $u=v=1_{X\oplus Y}$ in
\eqref{lin:eq:split-relation} also gives, in $\mathsf Q$,
\[
 1_{X\oplus Y}
   =[\iota_1][p_1]+[\iota_2][p_2].
\]
Suppress brackets on these coordinate arrows. For arrows
$a:Z\to X$ and $b:Z\to Y$ in $\mathsf Q$, the arrow
$\iota_1a+\iota_2b$ has components $a,b$. Conversely, if
$h:Z\to X\oplus Y$ has these components, then
\[
 h=(\iota_1p_1+\iota_2p_2)h
   =\iota_1a+\iota_2b.
\]
Dually, $cp_1+dp_2$ is the unique copairing of
$c:X\to Z$ and $d:Y\to Z$. Thus $X\oplus Y$ is a biproduct in
$\mathsf Q$. Its canonical addition is the given hom-monoid
addition, since
\[
 (p_1+p_2)(\iota_1a+\iota_2b)=a+b
 \qquad(a,b:Z\to X),
\]
where the coordinate maps belong to $X\oplus X$.

Extend the sum operation to the new arrows by setting
\[
 a\oplus b=\iota'_1ap_1+\iota'_2bp_2
 \qquad(a:X\to X',\ b:Y\to Y'),
\]
where primes refer to the target sum. The coordinate equations
show that this operation is functorial and preserves identities.
It agrees with each original $f\oplus g$, since the two arrows have
the same restrictions to the canonical inclusions. The images of
the original sum associators and unitors preserve the ordered
inclusions, so they are the canonical coproduct associators and
unitors.

To retain the original specified product objects, transport the
biproduct $X\oplus Y$ along the isomorphism
$q(i_{X,Y}):X\oplus Y\to X\otimes Y$. Its transported projections
are
\[
 q(p_r)q(i_{X,Y})^{-1}=q(\pi_r),
\]
so $X\otimes Y$, with the images of its original projections, is a
product in $\mathsf Q$. On arrows the product operation is
\[
 a\otimes b
   =q(i_{X',Y'})(a\oplus b)q(i_{X,Y})^{-1}.
\]
Naturality of the original comparison shows that this agrees with
$q(f\otimes g)$ for original arrows $f,g$. The original product
unit remains a zero object, and the images of the original product
associators and unitors preserve the ordered projections. They
are therefore the canonical product associators and unitors.
Consequently $\mathsf Q$, with these specified coproduct and product
structures, is bicartesian diagolinear, and $q$ preserves both
original specified structures strictly. Its identity comparison is
$q(i)$, since the entries of this arrow are identities on the
diagonal and zero off the diagonal.

Evaluating the arrow terms in the construction of $\Bic(\C)$
at their corresponding arrows, pairings and copairings in
$\mathsf Q$ therefore gives a functor
\[
 R:\Bic(\C)\longrightarrow\mathsf Q,
 \qquad Rb_{\C}=q.
\]
The defining relations of the construction hold in $\mathsf Q$,
and $R$ preserves the specified products and coproducts.

Conversely, $\Bic(\C)$ is linear by the first part of the proof.
The images of $\iota_r,p_r$ are its biproduct coordinate maps
for $X\oplus Y$. Their identity resolution implies that
\[
 b_{\C}(vu)
 =b_{\C}(v\iota_1p_1u)+b_{\C}(v\iota_2p_2u).
\]
Thus evaluation defines a functor
\[
 E:\mathsf Q\longrightarrow\Bic(\C),
 \qquad
 \sum_k[f_k]\longmapsto\sum_k b_{\C}(f_k).
\]
It respects the zero relations and
\eqref{lin:eq:split-relation}, and it is a functor by bilinearity
of composition. It preserves the biproduct coordinate maps and
zero objects. It also sends each transporting isomorphism $q(i)$
to $b_{\C}(i)$, so it preserves the original specified products
and coproducts.

Functors preserving biproducts preserve their canonical addition,
as follows from the expression of addition using the diagonal and
codiagonal. Hence both $E$ and $R$ are additive. Their composites
fix every original arrow. Every arrow of $\mathsf Q$ is a finite
sum of original generators by construction, and every arrow of
$\Bic(\C)$ is such a sum by the first part of the proof.
Consequently
\[
 RE=1_{\mathsf Q},
 \qquad
 ER=1_{\Bic(\C)}.
\]
This proves the asserted presentation and shows that no further
relations are required.
\end{proof}

\label{bic:sec:products}

The products considered here are products of structured categories,
not products of objects inside one category. A small family of
diagolinear categories has the componentwise sum and product
structures, the tuples of their units, and the componentwise identity
comparison. Joint epimorphicity of sum inclusions and joint monicity
of product projections are checked coordinatewise. A functor into
this product is strong for both structures if and only if all its
coordinate functors are strong. Natural transformations are formed
componentwise as well. These are therefore products in $\DiagLin$.
The same construction gives products in $\BiCart$, since genuine
finite coproducts and products are also formed componentwise. The
empty product is the category with one object and one arrow.

\begin{theorem}\label{bic:thm:finite-products}\label{lin:thm:finite-products}
Bicartesian reflection preserves finite products. More precisely, for
diagolinear categories $\C$ and $\D$, the canonical comparison
\[
 K_{\C,\D}:\Bic(\C\times\D)
       \longrightarrow\Bic(\C)\times\Bic(\D)
\]
is a natural isomorphism for the object-preserving construction above.
Its underlying functor is the identity on pairs of objects. The empty
product is preserved too.
\end{theorem}
\begin{proof}
Put $\E=\Bic(\C\times\D)$. The coordinate embeddings and
projections are
\[
\begin{aligned}
 e_{\C}:\C&\longrightarrow\C\times\D,
 &e_{\C}(X)&=(X,0_{\D}),\\
 e_{\D}:\D&\longrightarrow\C\times\D,
 &e_{\D}(Y)&=(0_{\C},Y),\\
 q_{\C}:\C\times\D&\longrightarrow\C,
 &q_{\D}:\C\times\D&\longrightarrow\D.
\end{aligned}
\]
All four functors preserve both specified structures. For an
embedding, its comparison maps in the zero coordinate are the
canonical isomorphisms between zero objects. Define
\[
 E_{\C}=\Bic(e_{\C}),\qquad E_{\D}=\Bic(e_{\D}),\qquad
 K=(K_{\C},K_{\D})=(\Bic(q_{\C}),\Bic(q_{\D})).
\]
The functor $K$ is the stated comparison, and satisfies the
commuting triangle
\[
\xymatrix@C=3em@R=3em{
 \C\times\D\ar[r]^{b_{\C\times\D}}
   \ar[dr]_{b_{\C}\times b_{\D}}
 &\Bic(\C\times\D)\ar[d]^K\\
 &\Bic(\C)\times\Bic(\D).
}
\]
It is the identity on object pairs because reflection retains all
objects. The images of the original unitors give an isomorphism
\begin{equation}\label{bic:eq:coordinate-decomposition}
 \theta_{X,Y}:E_{\C}X\oplus E_{\D}Y
 =(X\oplus0_{\C},0_{\D}\oplus Y)
 \xrightarrow{(\rho_X,\lambda_Y)}(X,Y)
\end{equation}
in $\E$.

We first prove fullness. Let $f:X\to X'$ be an arrow of
$\Bic(\C)$ and $g:Y\to Y'$ an arrow of $\Bic(\D)$. Define
\[
 h=\theta_{X',Y'}
       (E_{\C}f\oplus E_{\D}g)\theta_{X,Y}^{-1}.
\]
Then $Kh=(f,g)$. Indeed, $K_{\C}E_{\C}$ is the identity
functor because $q_{\C}e_{\C}$ is the identity. Also
$K_{\C}E_{\D}$ is the constant-zero functor. To justify the latter
statement on newly adjoined arrows as well, the constant-zero functor
between the reflected categories extends the original constant-zero
functor $q_{\C}e_{\D}$. Uniqueness of extension therefore identifies
it with $\Bic(q_{\C}e_{\D})$. Consequently, applying $K_{\C}$
to the formula for $h$ gives
\[
 \rho_{X'}(f\oplus1_{0_{\C}})\rho_X^{-1}=f.
\]
Here the projection $q_{\C}$ and its reflected functor are strict
for the specified operations in the componentwise construction.
The displayed equation is naturality of the right unitor. The
second coordinate gives $g$ in the same way, using the left unitor.
Thus $K$ is full.

For faithfulness, there are original natural transformations
\[
 a:e_{\C}q_{\C}\Rightarrow1_{\C\times\D},
 \quad a_{X,Y}=(1_X,0_{0,Y}),
 \qquad
 c:e_{\D}q_{\D}\Rightarrow1_{\C\times\D},
 \quad c_{X,Y}=(0_{0,X},1_Y).
\]
Apply the 2-functor $\Bic$. Since its units leave objects unchanged,
the components of the resulting transformations are the images of
these same coordinate inclusions. We suppress $b_{\C\times\D}$
on their notation. For every arrow
$u:(X,Y)\to(X',Y')$ of $\E$, naturality gives the commuting
squares
\[
\xymatrix@C=3.5em@R=3em{
 E_{\C}X\ar[r]^{a_{X,Y}}\ar[d]_{E_{\C}(K_{\C}u)}
 &(X,Y)\ar[d]^u
 &E_{\D}Y\ar[l]_{c_{X,Y}}\ar[d]^{E_{\D}(K_{\D}u)}\\
 E_{\C}X'\ar[r]_{a_{X',Y'}}
 &(X',Y')
 &E_{\D}Y'\ar[l]^{c_{X',Y'}}.
}
\]
If $Ku=Kv$, these squares imply
$u a_{X,Y}=v a_{X,Y}$ and $u c_{X,Y}=v c_{X,Y}$.
Under~\eqref{bic:eq:coordinate-decomposition}, the two arrows
$a_{X,Y},c_{X,Y}$ are the canonical inclusions of the sum structure.
They are jointly epimorphic in $\E$; indeed, this sum is now a
coproduct. Hence $u=v$. Thus $K$ is faithful.

Being full, faithful and bijective on objects, $K$ is an isomorphism
of underlying categories. Its inverse preserves the specified
structures via the inverse transported comparison isomorphisms.
Naturality follows from the construction of $K$ by applying the
2-functor $\Bic$ to the coordinate projections; the same observation
includes compatibility with natural transformations. Finally, the
one-object, one-arrow category is already bicartesian diagolinear,
so its reflection is isomorphic to itself by
Theorem~\ref{bic:thm:reflection}. This proves preservation of the
empty product, and the binary result gives every finite product.
\end{proof}

\section{Examples and Counterexamples}

\subsection{A non-pointed duoidal category of sets}\label{duo:ex:sets-not-diagolinear}

\begin{example}
The category $\Set$, with disjoint union as sum and cartesian product
as product, is a duoidal sum-product category which is not
diagolinear. Its units are the empty set $0=\varnothing$ and a
singleton $1=\{*\}$. Inclusions into a disjoint union are jointly
epimorphic, and cartesian projections are jointly monomorphic.
Every matrix of functions has a realisation by the coproduct and
product universal properties. In particular, the matrices
\eqref{duo:eq:block-matrix} are realisable, so
Theorem~\ref{duo:thm:matrix-interchange} gives a unique duoidal
structure. This is the usual coproduct--product example;
see~\cite[Example~2.2]{ShapiroSpivak2022}.

Explicitly, its interchange
\[
 \gamma:(A\times B)\sqcup(C\times D)
 \longrightarrow(A\sqcup C)\times(B\sqcup D)
\]
is specified on the two summands by
\[
\begin{aligned}
 \gamma\iota_1(a,b)&=(\iota_1(a),\iota_1(b)),\\
 \gamma\iota_2(c,d)&=(\iota_2(c),\iota_2(d)),
\end{aligned}
\]
where the inclusions on the right have the indicated target
summands. The remaining structure maps are the unique functions
\[
 \delta:\varnothing\longrightarrow\varnothing\times\varnothing,
 \qquad \mu:1\sqcup1\longrightarrow1,
 \qquad \tau:\varnothing\longrightarrow1.
\]
The map $\tau$ is not invertible, so this duoidal structure is not
normal. Moreover, there cannot be an identity comparison: its
component at $(1,\varnothing)$ would be a function
\[
 1\sqcup\varnothing\longrightarrow1\times\varnothing,
\]
with nonempty source and empty target. Thus duoidality does not force
pointedness or diagolinearity, even though the existence of an
identity comparison does force pointedness.
\end{example}

Example~\ref{duo:ex:sets-not-diagolinear} shows that a duoidal
sum-product category need not be diagolinear. The ordinal-sum
example below shows that a diagolinear category need not be duoidal.
A second choice of structures on finite posets and an example of
finite pointed metric spaces are neither bicartesian nor equipped
with invertible identity comparisons; the former is duoidal and the
latter is not. 

\subsection{A non-duoidal non-symmetric diagolinear category of finite posets and its bicartesian reflection}
\label{duo:sec:ordinal-posets}

Let $\mathsf P_{\mathrm{ord}}$ be the category of finite posets and
partial injective monotone maps. The domain of a partial map is an
arbitrary subset with its induced order, and composition is ordinary
composition of partial maps. For finite posets $P,Q$, their
\emph{ordinal sum} $P\triangleleft Q$ is their disjoint union with
the original orders within each summand and with every element of
$P$ below every element of $Q$. Equip this category with
\[
 P\oplus Q=P\otimes Q=P\triangleleft Q.
\]
Both operations thus have the empty poset as unit.
For ordinal sum on ordinary posets and total monotone maps, see
\cite[Example~2.11]{ShapiroSpivak2022}; here we use partial injective
monotone maps.

\begin{theorem}\label{duo:thm:ordinal-posets-not-duoidal}
The category $\mathsf P_{\mathrm{ord}}$ is diagolinear with identity
comparison. Its common monoidal structure admits no braiding, and
the two specified structures admit no duoidal structure.
\end{theorem}
\begin{proof}
Ordinal sum acts on arrows by disjoint union of their underlying
partial functions. This is functorial. It preserves monotonicity:
within a summand this follows from monotonicity of the given arrow,
and between summands it follows because the first target summand
lies below the second. It also preserves injectivity. The canonical
reassociations and unit identifications are order isomorphisms, so
they give a monoidal structure with zero object $\varnothing$ as unit.

Its canonical sum inclusions are the total inclusions of the two
summands. They are jointly epimorphic, since the two restrictions
of a partial function determine both its domain and its values.
Its canonical product projections are the partial identities defined
on the respective summands. These are monotone and jointly
monomorphic: composition with the two projections determines which
points map into each target summand and their values there. Hence
ordinal sum is both a monoidal sum structure and a monoidal product
structure. The identity natural transformation is compatible with
their common unitors and has identity matrix. It is therefore an
identity comparison, making $\mathsf P_{\mathrm{ord}}$ diagolinear.

Let $S$ be a singleton and $T$ a two-element antichain. The poset
$S\triangleleft T$ has one minimal and two maximal elements,
whereas $T\triangleleft S$ has two minimal and one maximal element.
An isomorphism in this category must be a total bijection with
monotone inverse, and hence an order isomorphism. These two posets
are not isomorphic. Thus ordinal sum admits no braiding. Since the
identity comparisons are invertible,
Theorem~\ref{duo:thm:invertible-comparison-symmetry} rules out a
duoidal structure.

One can also see the obstruction directly from the interchange
matrix. Taking all four objects to be singletons, its realisation
would have to carry the four labelled points in the order
$A,B,C,D$ to their copies in the order $A,C,B,D$. The matrix
conditions force this map to be total and to fix each label.
It would therefore send the four positions to $1,3,2,4$, which is
not monotone. Thus the matrix in~\eqref{duo:eq:block-matrix}
has no realisation for these objects.
\end{proof}



The \emph{incomparability graph} of a poset $R$ has
the elements of $R$ as vertices, with an edge between two distinct
vertices precisely when they are incomparable. A one-element graph
is connected, and the empty graph has no connected components.

\begin{lemma}\label{duo:lem:ordinal-components}
The connected components of the incomparability graph of a finite
poset $R$ admit a unique ordering $R_1,\ldots,R_k$ for which
\[
 R\cong R_1\triangleleft\cdots\triangleleft R_k,
\]
where each $R_j$ carries its induced order. In particular, these
components are exactly the nonempty factors in the unique
decomposition of $R$ into ordinal sums of posets with connected
incomparability graph.
\end{lemma}
\begin{proof}
Any two elements in different graph components are comparable,
since otherwise they would be joined by an edge. Let $x<y$ lie
in different components. If $x'$ is joined to $x$ by an
incomparability edge, then $x'$ is comparable with $y$. The
possibility $y<x'$ would imply $x<x'$, contradicting that edge.
Thus $x'<y$. Repeating this argument along an incomparability path
shows that every element in the component of $x$ lies below $y$.
The same argument, with the roles reversed, shows that every such
element lies below every element in the component of $y$.

Consequently comparison between distinct components has a uniform
direction. These directions define a strict total order on the
components: any two are comparable, both directions cannot hold,
and transitivity follows from transitivity in $R$. Listing the
components in this order gives the asserted ordinal sum, and its
ordering is forced by the comparisons in $R$.

Finally, an ordinal sum has no incomparability edge between
different summands. Therefore a decomposition into nonempty
summands with connected incomparability graphs must have precisely
these graph components as its summands. This proves uniqueness,
including the empty decomposition when $R=\varnothing$.
\end{proof}

For finite posets $P,Q$, let $\mathcal J(P,Q)$ be the set of
nonzero partial injective monotone maps $f:P\rightharpoonup Q$
whose image, with its induced order, has connected incomparability
graph. Let $\langle f\rangle$ denote the generator corresponding
to $f\in\mathcal J(P,Q)$ in the free commutative monoid
$\N[\mathcal J(P,Q)]$. For any original arrow
$f:P\rightharpoonup Q$ and any incomparability component $K$ of
$\im f$, write $f_K$ for its restriction to $f^{-1}(K)$, still
regarded as an arrow from $P$ to $Q$. Put
\begin{equation}\label{duo:eq:ordinal-normal-form}
 \nu(f)=\sum_{K}\langle f_K\rangle
       \in\N[\mathcal J(P,Q)],
\end{equation}
where the sum runs over those components; in particular, $\nu(0)=0$.

\begin{theorem}\label{bic:thm:ordinal-poset-reflection}
The bicartesian reflection of $\mathsf P_{\mathrm{ord}}$ is linear,
and its hom-monoids have the description
\begin{equation}\label{duo:eq:ordinal-reflection-homs}
 \Bic(\mathsf P_{\mathrm{ord}})(P,Q)
       \cong\N[\mathcal J(P,Q)].
\end{equation}
Under these isomorphisms the reflection unit sends $f$ to $\nu(f)$.
Composition is the bilinear extension of
\begin{equation}\label{duo:eq:ordinal-reflection-composition}
 \langle g\rangle\langle f\rangle=\nu(gf),
\end{equation}
and the identity on $P$ is $\nu(1_P)$. The reflection unit is
faithful, and the reflection is duoidal. The full subcategory of the reflection on finite chains is
equivalent to $\Mat_{\N}$. The whole reflection is not equivalent
to $\Mat_{\N}$: for a two-element antichain $T$, its additive
monoid of endomorphisms is the free commutative monoid $\N^6$.
\end{theorem}
\begin{proof}
Since the two specified tensors coincide and the identity comparison
is the identity, Lemma~\ref{bic:prop:common-tensor-presentation}
shows that the reflection is linear and gives a presentation of its
hom-monoids. We verify that~\eqref{duo:eq:ordinal-normal-form}
is exactly a normal form for that presentation.

Consider a generating relation obtained from original arrows
\[
 P\xrightarrow{u}X\triangleleft Y\xrightarrow{v}Q.
\]
Write
\[
 h=vu,\qquad
 h_X=v\iota_1\pi_1u,\qquad
 h_Y=v\iota_2\pi_2u.
\]
The domains of $h_X$ and $h_Y$ partition the domain of $h$,
and $h$ is their union. Their images are disjoint because $v$
is injective. Moreover, every point of $\im h_X$ lies strictly
below every point of $\im h_Y$: the corresponding points in
$X\triangleleft Y$ have this relation, and $v$ is monotone and
injective on its domain. With induced orders, therefore,
\[
 \im h=(\im h_X)\triangleleft(\im h_Y),
\]
with the empty factors omitted. There are no incomparability edges
between the two factors. Its incomparability components are thus
precisely those of the two factors, and
\[
 \nu(h)=\nu(h_X)+\nu(h_Y).
\]
Consequently the additive extension of $\nu$ respects every relation
\eqref{lin:eq:split-relation}, as well as $[0]=0$.

Conversely, let the ordered incomparability components of $\im f$
be $I_1,\ldots,I_k$. By Lemma~\ref{duo:lem:ordinal-components},
$f$ factors in the original category as
\[
\xymatrix@C=4em@R=3em{
 P\ar[r]^{\widetilde f}\ar[dr]_f
 &I_1\triangleleft\cdots\triangleleft I_k\ar[d]^v\\
 &Q,
}
\]
where $\widetilde f$ is $f$ with codomain restricted to its image,
using the canonical ordinal-sum identification, and $v$ is the
inclusion of that image into $Q$. Iterating the split relation
gives
\begin{equation}\label{duo:eq:ordinal-split-components}
 [f]=\sum_{j=1}^k[f_{I_j}].
\end{equation}
For $k=0$ this is $[0]=0$; for $k=1$ there is no split to perform.
Each arrow on the right belongs to $\mathcal J(P,Q)$.

The homomorphism $\nu$ therefore descends from the presented
hom-monoid to $\N[\mathcal J(P,Q)]$. In the reverse direction,
send $\langle g\rangle$ to the class $[g]$ in the presentation.
The first composite fixes every generator $\langle g\rangle$,
since its image is incomparability-connected. The other composite
fixes every original generator $[f]$ by
\eqref{duo:eq:ordinal-split-components}. These homomorphisms are
inverse, proving~\eqref{duo:eq:ordinal-reflection-homs} and showing
that there are no further additive relations.

We also check explicitly that the normal forms respect composition.
Suppose $\im f=I_1\triangleleft\cdots\triangleleft I_k$ is its
component decomposition. For any composable original arrow $a$,
the maps $f_{I_j}a$ are restrictions whose union is $fa$, and
their nonempty images form an ordinal-sum decomposition of
$\im(fa)$, since they are subsets of the ordered blocks $I_j$.
The incomparability-component decomposition of each such subset
therefore gives
\[
 \nu(fa)=\sum_j\nu(f_{I_j}a).
\]
For an original arrow $b$ composable on the other side, the images
of $bf_{I_j}$ are again disjoint ordered blocks: $b$ is injective
and preserves every comparison between the surviving points of
different $I_j$. Their union is $\im(bf)$, so similarly
\[
 \nu(bf)=\sum_j\nu(bf_{I_j}).
\]
Thus decomposition may be performed before or after either
composition. In particular, if the normal forms of $f$ and $g$
are sums of restrictions $f_j$ and $g_l$, then
\[
 \nu(gf)=\sum_{j,l}\nu(g_lf_j).
\]
This is precisely compatibility with the bilinear rule
\eqref{duo:eq:ordinal-reflection-composition}. It also verifies
associativity directly: either parenthesisation of the product of
three generators gives $\nu(hgf)$. The same equations show that
$\nu(1_P)$ and $\nu(1_Q)$ act as the appropriate identities.
These operations agree with those of the presented category, since
the reflection unit preserves original compositions.

The normal form of an original arrow is a sum of its restrictions
to disjoint inverse images of image components. These restrictions
recover the original arrow by their union, including its domain of
definition. Equality of two such formal sums therefore implies
equality of the original arrows. The unit is faithful. The
reflection is duoidal because it is linear, or directly by
Lemma~\ref{duo:lem:bicartesian-duoidal}.

If $Q$ is a finite chain, every induced subposet of $Q$ is a chain,
so its incomparability components are singletons. For finite chains
$P,Q$, the generators in $\mathcal J(P,Q)$ are consequently
exactly the one-point partial maps. Choosing the usual chains
$\{1,\ldots,n\}$, including the empty chain, these generators
are matrix units. Their products are zero when the intermediate
points differ and the corresponding matrix unit when they agree.
Bilinear composition is therefore ordinary matrix multiplication
over $\N$, proving the assertion about the full subcategory of
finite chains.

Finally, let $T$ be the two-element antichain. Every nonempty
subset of $T$ has connected incomparability graph, and every
partial injection from $T$ to itself is monotone. There are four
one-point partial maps and two total bijections. These are all the
generators of its additive monoid of endomorphisms, so
\[
 \Bic(\mathsf P_{\mathrm{ord}})(T,T)\cong\N^6.
\]
An equivalence between linear categories preserves their canonical
addition: it preserves zero objects and biproduct diagrams, and
addition is recovered from diagonals and codiagonals. If the whole
reflection were equivalent to $\Mat_{\N}$, the object $T$ would
therefore have to correspond to some $n$, giving an isomorphism of
commutative monoids $\N^6\cong\N^{n^2}$. A free commutative
monoid on $r$ generators has exactly $r$ nonzero elements which
cannot be written as a sum of two nonzero elements. Hence such an
isomorphism would imply $6=n^2$, which is impossible. This proves
that the full reflection is not equivalent to $\Mat_{\N}$.
\end{proof}

Thus this reflection supplies a symmetry and a duoidal interchange
while preserving all distinctions between original arrows. Its
finite-chain part is the familiar category of matrices over $\N$,
but the other finite posets contribute additional generators and
objects.

\subsection{A non-symmetric duoidal  category of finite posets}
\label{duo:sec:finite-posets}

Retain the underlying category of finite posets and partial injective
monotone maps from Section~\ref{duo:sec:ordinal-posets}, but now write
$\mathsf P$ for it equipped with
\[
 P\oplus Q=P\sqcup Q,\qquad
 P\otimes Q=P\triangleleft Q.
\]
Here $P\sqcup Q$ has no comparisons between the two summands.
Both operations act on arrows by disjoint union of their underlying
partial functions, and both have the empty poset as unit.
For ordinary posets and total monotone maps, these two operations
and the interchange below occur in
\cite[Example~2.11]{ShapiroSpivak2022}.

\begin{theorem}\label{duo:thm:partial-monotone-injections}
The specified structures make $\mathsf P$ a duoidal diagolinear
category. Its sum structure is symmetric, but its product structure
admits no braiding. The specified sum is not a coproduct, the
specified product is not a product, and the identity comparison
is not always invertible. In particular,
$\mathsf P$ is neither bicartesian nor a diagolinear category with
invertible identity comparisons.
\end{theorem}
\begin{proof}
The ordinal-sum monoidal structure was verified in the proof of
Theorem~\ref{duo:thm:ordinal-posets-not-duoidal}. Disjoint union is
also a bifunctor: partial injectivity and monotonicity can be checked
separately on each summand. Its usual reassociations and unit
identifications are order isomorphisms and satisfy the monoidal
axioms. The common unit $\varnothing$ is a zero object.

The disjoint-union inclusions are jointly epimorphic, since the
restrictions to the two summands determine the domain and values
of a partial function. The ordinal-sum projections are the partial
identities on the respective summands. As in the preceding proof,
they are monotone and jointly monomorphic. Thus we have a monoidal
sum structure and a monoidal product structure. The total identity
on the underlying disjoint union gives an injective monotone map
\[
 i_{P,Q}:P\sqcup Q\longrightarrow P\triangleleft Q.
\]
These maps are natural, respect the unit identifications, and have
identity matrices. They therefore form the identity comparison.

Define the interchange by the canonical shuffle
\begin{equation}\label{duo:eq:poset-interchange}
 (A\triangleleft B)\sqcup(C\triangleleft D)
 \longrightarrow
 (A\sqcup C)\triangleleft(B\sqcup D),
\end{equation}
sending each element to its copy in the correspondingly labelled
summand. This is a total injective monotone map: the comparisons
between different source blocks are $A<B$ and $C<D$, and both
hold in the target, which additionally has $A<D$ and $C<B$.
Testing its two source inclusions and two target projections gives
the four entries of~\eqref{duo:eq:block-matrix}. Hence
Theorem~\ref{duo:thm:matrix-interchange} supplies the duoidal
structure, including its naturality and coherence equations.

The usual interchange of disjoint-union summands is an order
isomorphism and gives the sum symmetry. The product admits no
braiding by Theorem~\ref{duo:thm:ordinal-posets-not-duoidal}, since
its monoidal structure and underlying category have not changed.

Let $S$ be a singleton. The row $(1_S,1_S)$ has no realisation
$S\sqcup S\rightharpoonup S$: such a partial function would be
defined on both elements and identify them, contrary to injectivity.
The column $(1_S,1_S)^{\mathsf T}$ has no realisation
$S\rightharpoonup S\triangleleft S$, since the two partial
projections have disjoint domains and the image of the single
element cannot belong to both. Thus the specified structures are
not a coproduct and a product, respectively.

Finally, $i_{S,S}$ is the total bijection from a two-element
antichain to a two-element chain which preserves the two labels.
An inverse in the category of partial maps would have to be its
ordinary inverse function. That function sends the comparable
points of the chain to incomparable points, so it is not monotone.
Consequently $i_{S,S}$ is not invertible.
\end{proof}

\subsection{A non-duoidal symmetric diagolinear category of pointed metric spaces}
\label{duo:sec:metric-spaces}

The $\ell^2$ product metric used below is standard; see
\cite[p.~193]{BorsikDobos1981}. The usual metric wedge has
cross-distance $a+b$, as in \cite[Proposition~2.12]{Ishiki2023}.
Here we modify this cross-distance using the function $F$ below,
and verify the required properties directly.

\begin{samepage}
Let $\mathsf{Met}_*$ be the category of finite pointed metric spaces
and base-point-preserving nonexpansive maps. Write $0$ for a base
point and put $r_X(x)=d_X(x,0)$. For nonnegative real numbers define
\begin{equation}\label{duo:eq:metric-cross-distance}
 F(a,b)=\max\left\{\sqrt{a^2+b^2},\,\min\{a+b,1\}\right\}.
\end{equation}
\end{samepage}
We use the cartesian product of pointed sets for $X\otimes Y$, with
the metric
\[
 d_{X\otimes Y}\bigl((x,y),(x',y')\bigr)
   =\sqrt{d_X(x,x')^2+d_Y(y,y')^2}.
\]
For $X\oplus Y$ we use the wedge of pointed sets. Its metric restricts
to the given metric on each summand, while the distance between
points in different summands is
\[
 d_{X\oplus Y}(\iota_1x,\iota_2y)=F(r_X(x),r_Y(y)).
\]
Both operations act on maps by the usual cartesian product and wedge
of pointed maps, respectively.

\begin{theorem}\label{duo:thm:metric-not-duoidal}
These operations make $\mathsf{Met}_*$ a diagolinear category with
both specified monoidal structures symmetric. The identity comparison
is the canonical inclusion of the wedge into the cartesian product,
and is not always invertible. These structures are not bicartesian
and do not admit a duoidal structure.
\end{theorem}

\begin{proof}
We first verify the metric and monoidal assertions. The function $F$
is symmetric and increasing in each variable, and satisfies
\begin{equation}\label{duo:eq:metric-cross-bounds}
 F(a,0)=a,\qquad
 \max\{a,b\}\leq\sqrt{a^2+b^2}\leq F(a,b)\leq a+b.
\end{equation}
It is also $1$-Lipschitz in each variable separately. Indeed,
$a\mapsto\sqrt{a^2+b^2}$ and $a\mapsto\min\{a+b,1\}$ are
$1$-Lipschitz, and the maximum of two such real-valued functions
is again $1$-Lipschitz. The other variable is treated in the same way.

The equation $F(a,0)=a$ shows that the proposed wedge metric is
well defined at the identified base points. Symmetry and positivity
follow from~\eqref{duo:eq:metric-cross-bounds}. For the triangle
inequality, only triangles meeting both summands require verification.
Suppose $x,x'$ lie in $X$ and $y$ lies in $Y$. The separate Lipschitz
property and the reverse triangle inequality in $X$ give
\[
\begin{aligned}
 F(r_X(x),r_Y(y))
 &\leq |r_X(x)-r_X(x')|+F(r_X(x'),r_Y(y))\\
 &\leq d_X(x,x')+F(r_X(x'),r_Y(y)).
\end{aligned}
\]
For the remaining orientation of this triangle,
\[
 d_X(x,x')\leq r_X(x)+r_X(x')
 \leq F(r_X(x),r_Y(y))+F(r_X(x'),r_Y(y)).
\]
Interchanging $x,x'$ and interchanging the two summands proves all
the required inequalities. Thus $X\oplus Y$ is a metric space.
The product formula is a metric by the coordinate triangle
inequalities and the Euclidean triangle inequality in $\mathbb R^2$.

If $f:X\to X'$ and $g:Y\to Y'$ are nonexpansive and preserve base
points, then $r_{X'}(f(x))\leq r_X(x)$ and
$r_{Y'}(g(y))\leq r_Y(y)$. Monotonicity of $F$ therefore proves
nonexpansiveness of $f\oplus g$ on pairs from different summands;
on each individual summand it follows from the hypothesis.
The product metric similarly makes $f\otimes g$ nonexpansive.
The underlying pointed-set operations preserve identities and
composition, so both operations are bifunctors.

The one-point space is a zero object. It is a unit for both
operations, with the usual identifications as isometries. In an
iterated wedge, the distance of a point to the base point is unchanged.
Consequently the distance between two points in the same original
summand is its original distance, and between points in distinct
original summands it is $F$ of their original distances to the base
point. These descriptions are independent of the bracketing.
Thus the ordinary wedge reassociations are isometries. No
associativity property of the scalar function $F$ is needed for
this conclusion. For iterated products the metric is the square root
of the sum of the squared coordinate distances, again independently
of the bracketing. The usual reassociations, unit identifications
and exchanges of factors are therefore isometries for the respective
operations. Their coherence equations hold because they hold for
the underlying pointed-set maps. This gives the two symmetric
monoidal structures.

The wedge inclusions are jointly epimorphic: every point of a wedge
belongs to one of its two summands. The cartesian projections are
nonexpansive and jointly monomorphic: together they determine every
underlying map into a cartesian product. Hence these structures are
a monoidal sum structure and a monoidal product structure.
The canonical pointed map
\[
 i_{X,Y}:X\oplus Y\longrightarrow X\otimes Y,
 \qquad \iota_1x\longmapsto(x,0),\quad
        \iota_2y\longmapsto(0,y),
\]
is nonexpansive. Within one summand it preserves distances; on a
pair from different summands the target distance is
$\sqrt{r_X(x)^2+r_Y(y)^2}$, which is bounded by the source distance
by~\eqref{duo:eq:metric-cross-bounds}. Its underlying pointed map is
natural and compatible with both units, and its matrix is the
identity matrix. This proves diagolinearity.

Let $S=\{0,s\}$ with $d_S(0,s)=1/2$. The comparison
$i_{S,S}:S\oplus S\to S\otimes S$ has a three-point source
and a four-point target. It misses $(s,s)$, so it is not an
isomorphism. The specified product is not a categorical product:
if it were, the two identity maps of $S$ would have a pairing
$\Delta:S\to S\otimes S$. Its projections would force
$\Delta(s)=(s,s)$ and $\Delta(0)=(0,0)$, but
\[
 d_{S\otimes S}\bigl((0,0),(s,s)\bigr)
   =1/\sqrt2>1/2=d_S(0,s).
\]
Thus this pairing is not nonexpansive, proving that the specified
structures are not bicartesian.

To disprove duoidality, take $A=B=C=D=S$. Any realisation of
\eqref{duo:eq:block-matrix} would have to send the two product
summands according to
\[
\begin{aligned}
 \gamma\iota_1(a,b)&=(\iota_1a,\iota_1b),\\
 \gamma\iota_2(c,d)&=(\iota_2c,\iota_2d).
\end{aligned}
\]
Indeed, composing with the two cartesian projections determines
these two formulas from the prescribed matrix entries. Put
$u=\iota_1(s,s)$ and $v=\iota_2(s,s)$. The distance of $(s,s)$
to the base point in $S\otimes S$ is $1/\sqrt2$, so
\[
 d(u,v)=F(1/\sqrt2,1/\sqrt2)=1.
\]
On the other hand, $F(1/2,1/2)=1$, and hence the required images
would have distance
\[
 d(\gamma u,\gamma v)
   =\sqrt{F(1/2,1/2)^2+F(1/2,1/2)^2}=\sqrt2.
\]
Thus the required realisation is not nonexpansive and cannot be a
morphism of $\mathsf{Met}_*$. The matrix criterion in
Theorem~\ref{duo:thm:matrix-interchange} rules out a duoidal structure.
\end{proof}

\subsection{A duoidal category of finite sets and its bicartesian reflection}
\label{duo:sec:finite-partial-maps}

The category of sets and partial injections with disjoint union is a
standard example of a unique decomposition category
\cite[Example~8]{HaghverdiScott2006}; here we restrict to finite sets.

Let $\mathbf{FinPInj}$ denote the category of finite
sets and partial injections. Choose disjoint union as the common tensor,
with the empty set as zero object. The canonical inclusions are jointly
epic, since a partial injection out of a disjoint union is determined by its
restrictions to the two summands. The partial projections are jointly monic,
since they determine both the domain of definition and the values of a
partial injection into the disjoint union. Thus this is a diagolinear
category when disjoint union is used for both specified structures
and $i$ is the identity.

This category is duoidal although it is not bicartesian with
these specified structures. Disjoint union is symmetric monoidal,
so the construction of \cite[Example~3.2]{BookerStreet2013} gives
the interchange
\[
 (A\sqcup B)\sqcup(C\sqcup D)
       \xrightarrow{\ \cong\ }(A\sqcup C)\sqcup(B\sqcup D),
\]
which sends each element to its copy in the correspondingly labelled
summand. This is a total bijection, hence a morphism in
$\FinPInj$. It realises
\eqref{duo:eq:block-matrix} and is the unique interchange by
Theorem~\ref{duo:thm:matrix-interchange}.

If $X$ is a singleton, the row $(1_X,1_X)$ has no realisation
$X\sqcup X\to X$ in $\FinPInj$: such a realisation would be
defined on both elements and send them to the same element,
contrary to injectivity. The column
$\colvec{1_X}{1_X}$ has no realisation $X\to X\sqcup X$,
since a partial function cannot send the single element into both
domains of the two partial projections. Thus this category is not
bicartesian. Its reflection, computed next, is linear and
duoidal; its interchange is the corresponding biproduct shuffle.

Write $\operatorname{Mat}_{\N}$ for the linear category whose objects are
nonnegative integers and whose arrows $n\to m$ are $m\times n$ matrices over
$\N$, with ordinary matrix composition. Equivalently, it is the category
of finite free commutative monoids and their homomorphisms
\cite[Section~5.3]{Lack2004}.

\begin{theorem}\label{lin:prop:finite-partial-matrix-completions}
There is an equivalence of linear categories
\[
 \Bic(\mathbf{FinPInj})\simeq\operatorname{Mat}_{\N}.
\]
The unit of bicartesian reflection is faithful.
\end{theorem}

\begin{proof}
Work on the skeleton whose objects are
$\mathbf n=\{1,\ldots,n\}$, including $\mathbf0=\varnothing$. A partial injection
$f:\mathbf n\rightharpoonup\mathbf m$ has an incidence matrix $I_f$ with entries
\[
 (I_f)_{ji}=
 \begin{cases}
  1,&\text{if $f$ is defined at $i$ and $f(i)=j$},\\
  0,&\text{otherwise}.
 \end{cases}
\]
Each column and each row has at most one nonzero entry. Ordinary matrix multiplication gives
$I_{gf}=I_gI_f$: for a fixed source element, at most one intermediate value
can contribute. The incidence construction is faithful, preserves zero
arrows, and sends disjoint union and its canonical maps to the standard
biproduct diagrams. Lemma~\ref{bic:prop:invertible-comparison} makes the
bicartesian reflection linear. The incidence functor therefore extends
to a linear functor from that reflection to $\operatorname{Mat}_{\N}$. Write
$\C=\mathbf{FinPInj}$, $I$ for the incidence functor and
$\overline I$ for its extension. We obtain
\[
\xymatrix@C=4em@R=2.5em{
 \C\ar[r]^{b_{\C}}\ar[dr]_I
 &\Bic(\C)\ar@{.>}[d]^{\overline I}\\
 &\operatorname{Mat}_{\N} .
}
\]

In the original category, the only endomorphisms of $\mathbf1$ are the
identity and the zero partial injection. By Lemma~\ref{bic:prop:invertible-comparison}, every
endomorphism of $\mathbf1$ in the bicartesian reflection has the form $n[\id_{\mathbf1}]$, for some
$n\in\N$. The extended incidence functor sends this arrow to the
$1\times1$ matrix $(n)$, so no two different values of $n$ are identified.
Hence
\[
 \End_{\Bic(\C)}(\mathbf1)\cong\N.
\]

Every $\mathbf n$ is the biproduct of $n$ copies of $\mathbf1$ in the bicartesian reflection.
Thus an arrow $\mathbf n\to\mathbf m$ is uniquely determined by its $mn$ entries in
$\End(\mathbf1)=\N$, and composition is matrix multiplication. More explicitly,
if $e_{ji}:\mathbf n\rightharpoonup\mathbf m$ is the partial injection defined only at
$i$, with value $j$, then the matrix $(a_{ji})$ is realised by
\[
 \sum_{j=1}^{m}\sum_{i=1}^{n}a_{ji}[e_{ji}].
\]
The extended incidence functor is therefore fully faithful and essentially
surjective. Finally, its composite with the unit is the faithful original
incidence functor, so the unit is faithful.
\end{proof}

The category $\FinPInj$ also provides us with a counterexample for the preservation of countable products by the bicartesian reflection.

\begin{theorem}\label{bic:prop:countable-products-fail}\label{lin:prop:countable-products-fail}
Bicartesian reflection need not preserve countable products. More
precisely, let $\mathcal P=\FinPInj$, equipped with disjoint union
as both its monoidal structures and with identity comparison. The
canonical comparison
\[
 K:\Bic(\mathcal P^{\N})\longrightarrow\Bic(\mathcal P)^{\N}
\]
is not full and therefore is not an equivalence.
\end{theorem}
\begin{proof}
Use the skeleton of finite sets with objects
$\mathbf n=\{1,\ldots,n\}$, including $\mathbf0=\varnothing$.
The diagolinear structure on $\mathcal P$ is verified in
Section~\ref{duo:sec:finite-partial-maps}. Both $\mathcal P$ and
$\mathcal P^{\N}$ have invertible identity comparisons. By
Lemma~\ref{bic:prop:invertible-comparison}, their bicartesian
reflections are linear, and every arrow of either reflection is a
finite sum of images of original arrows. Write $[f]=b(f)$ for these
images. The computation in
Theorem~\ref{lin:prop:finite-partial-matrix-completions} gives
\begin{equation}\label{bic:eq:singleton-endomorphisms}
 \End_{\Bic(\mathcal P)}(\mathbf1)\cong\N,
 \qquad n[1_{\mathbf1}]\longmapsto n.
\end{equation}

Let $E=(\mathbf1,\mathbf1,\ldots)$ in $\mathcal P^{\N}$. Every original
endomorphism of $E$ is a sequence of zeros and identities. By the
finite-sum assertion of Lemma~\ref{bic:prop:invertible-comparison},
every endomorphism in $\Bic(\mathcal P^{\N})$ has a representative
\[
 u=\sum_{r=1}^{m}[(\epsilon_{r,i})_{i\in\N}],
 \qquad \epsilon_{r,i}\in\{0,1\},
\]
for some finite $m$, including $m=0$. Since $K$ preserves biproducts,
it preserves addition. Equation~\eqref{bic:eq:singleton-endomorphisms}
therefore identifies its image with the sequence
\[
 K(u)=\left(\sum_{r=1}^{m}\epsilon_{r,i}\right)_{i\in\N}.
\]
Every entry is at most $m$. On the other hand, the endomorphism
monoid of $E$ in $\Bic(\mathcal P)^{\N}$ is $\N^{\N}$ and
contains the unbounded sequence $(1,2,3,\ldots)$. This endomorphism
cannot be in the image of $K$, so $K$ is not full. Relations between
different terms do not remove the obstruction: every arrow still
has at least one finite-sum representative, whose image is bounded.

In fact the image on this endomorphism monoid is exactly the set
of bounded sequences. If $a_i\leq M$ for all $i$, put
$\epsilon_{r,i}=1$ when $r\leq a_i$, and
$\epsilon_{r,i}=0$ otherwise, for $1\leq r\leq M$.
Each $(\epsilon_{r,i})_i$ is an original endomorphism of $E$, and
the sum of their images maps to $(a_i)_i$.
\end{proof}

The categories $\mathcal P$ and $\mathcal P^{\N}$ in the preceding
example are duoidal: their interchanges are the canonical shuffles,
formed coordinatewise in the latter case. Their reflections, and
$\Bic(\mathcal P)^{\N}$, are linear and hence duoidal.
\label{bic:sec:reflection-examples}

\subsection{A duoidal category of finite groups}
\label{duo:sec:finite-groups}

Let $\FinGrp$ have its ordinary direct product as both specified
monoidal structures. The trivial group is a zero object. The canonical
axes are jointly epimorphic because every element of $G\times H$ is
$(g,1)(1,h)$, and the ordinary product projections are jointly
monomorphic. With identity comparison this makes $\FinGrp$
diagolinear, as described in the abstract. Its bicartesian reflection
is linear by Lemma~\ref{bic:prop:invertible-comparison}.
The common tensor is symmetric and the identity comparison is
invertible, so Theorem~\ref{duo:thm:invertible-comparison-symmetry}
shows that the original category is duoidal. Its interchange is the group
isomorphism
\[
 (A\times B)\times(C\times D)\longrightarrow
 (A\times C)\times(B\times D),\qquad
 ((a,b),(c,d))\longmapsto((a,c),(b,d)).
\]
It is the shuffle supplied by the symmetric common tensor, as in
\cite[Example~3.2]{BookerStreet2013}. The linear reflection is
duoidal as well, and the unit preserves this interchange.
The following explicit obstruction shows that its unit is not faithful.
All groups and maps used in the proof are finite.

Let $C_2=\langle s\mid s^2=1\rangle$ be the cyclic group of order two.

\begin{lemma}\label{lin:lem:commuting-involution-addition}
Let $F:\mathbf{FinGrp}\to\A$ preserve finite products, where
$\A$ is linear.  If $x,y$ are commuting elements of a finite group
$G$ with $x^2=y^2=1$, let $f_x,f_y,f_{xy}:C_2\to G$ send the generator of
$C_2$ to $x,y,xy$, respectively.  Then
\[
 F(f_{xy})=F(f_x)+F(f_y).
\]
The identity element is allowed as one of $x,y$.
\end{lemma}

\begin{proof}
The map
\[
 c:C_2\times C_2\longrightarrow G,
 \qquad c(s^r,s^t)=x^ry^t,
\]
is a homomorphism with restrictions $f_x,f_y$ to the two axes, and
$c\Delta=f_{xy}$.  Since $F$ preserves products, it preserves the diagonal
and, since it preserves the trivial group, the canonical axes.  In
$\A$ the image of this product diagram is a biproduct diagram.
The two descriptions of $c$ are recorded in the commutative diagrams
\[
\xymatrix@C=2.5em@R=2.5em{
 C_2\ar[r]^{\iota_1}\ar[dr]_{f_x}
 &C_2\times C_2\ar[d]^c
 &C_2\ar[l]_{\iota_2}\ar[dl]^{f_y}\\
 &G
}
\qquad
\xymatrix@C=2.5em@R=2.5em{
 C_2\ar[r]^{\Delta}\ar[dr]_{f_{xy}}
 &C_2\times C_2\ar[d]^c\\
 &G .
}
\]
After applying $F$, the first diagram says that $F(c)$ is the copairing
of $F(f_x),F(f_y)$; the second identifies its composite with the
diagonal. Consequently
$F(c)F(\Delta)=F(f_x)+F(f_y)$.
\end{proof}

The matrix construction in the following proof uses the Pauli-operator
identities in Mermin's two-particle argument
\cite[p.~3374, equation~(4)]{Mermin1990}.
Their application to finite-product-preserving functors is proved here.

\begin{theorem}\label{lin:thm:finite-group-no-faithful-completion}
There is no faithful finite-product-preserving functor from
$\mathbf{FinGrp}$ to a linear category.  More precisely, there are a finite
group $G$ and a nonzero homomorphism $f_z:C_2\to G$ with central image such
that every such functor sends $f_z$ to zero.
\end{theorem}

\begin{proof}
Consider the real matrices
\[
 A=\begin{pmatrix}1&0\\0&-1\end{pmatrix},
 \qquad
 B=\begin{pmatrix}0&1\\1&0\end{pmatrix},
 \qquad
 Q=AB=\begin{pmatrix}0&1\\-1&0\end{pmatrix}.
\]
They satisfy
$A^2=B^2=I_2$, $BA=-AB$, and $Q^2=-I_2$.
Using the ordinary Kronecker product of matrices, set
\[
\begin{aligned}
 a=A\otimes I_2&=
 \begin{pmatrix}
 1&0&0&0\\0&1&0&0\\0&0&-1&0\\0&0&0&-1
 \end{pmatrix},
 &b=B\otimes I_2&=
 \begin{pmatrix}
 0&0&1&0\\0&0&0&1\\1&0&0&0\\0&1&0&0
 \end{pmatrix},\\[1ex]
 c=I_2\otimes A&=
 \begin{pmatrix}
 1&0&0&0\\0&-1&0&0\\0&0&1&0\\0&0&0&-1
 \end{pmatrix},
 &d=I_2\otimes B&=
 \begin{pmatrix}
 0&1&0&0\\1&0&0&0\\0&0&0&1\\0&0&1&0
 \end{pmatrix},
\end{aligned}
\]
and let $G$ be the subgroup of $\operatorname{GL}_4(\mathbb R)$ generated
by $a,b,c,d$.  These are signed permutation matrices, so $G$ is finite.
Each of $a,b,c,d$ is an involution, and each of $a,b$ commutes with each of
$c,d$.  Moreover,
\[
 z=(ab)^2=-I_4
\]
is a nonidentity central involution in $G$.

Set
\[
 t=(ac)(bd)=Q\otimes Q,
 \qquad
 t'=(ad)(bc)=Q\otimes(-Q)=-t=zt.
\]
The factors $ac$ and $bd$ commute, because
\[
 (bd)(ac)=(BA)\otimes(BA)=(-Q)\otimes(-Q)=t.
\]
Similarly $ad$ and $bc$ commute, since
$(bc)(ad)=(BA)\otimes(AB)=-t$.
All four factors $ac,bd,ad,bc$ are involutions.  Also
\[
 t^2=(Q^2)\otimes(Q^2)=I_4,
 \qquad (t')^2=I_4,
 \qquad t\ne t'.
\]

Let $F:\mathbf{FinGrp}\to\A$ preserve finite products, with
$\A$ linear.  Applying
Lemma~\ref{lin:lem:commuting-involution-addition} repeatedly gives
\begin{align*}
 F(f_t)
 &=F(f_{ac})+F(f_{bd})\\
 &=F(f_a)+F(f_c)+F(f_b)+F(f_d),\\
 F(f_{t'})
 &=F(f_{ad})+F(f_{bc})\\
 &=F(f_a)+F(f_d)+F(f_b)+F(f_c).
\end{align*}
Commutativity and associativity of addition give
$F(f_t)=F(f_{t'})$, although $f_t\ne f_{t'}$.  Thus $F$ is not faithful.

There is a stronger conclusion.  Put $q=F(f_t)$ and $r=F(f_z)$.  Since
$t'=zt$ and $z$ commutes with $t$, the lemma gives
\[
 q=F(f_{t'})=r+q.
\]
Applying it to the commuting pair $t,t$, and using $t^2=1$, gives
$q+q=0$.  Adding $q$ to both sides of $q=r+q$ now yields
\[
 0=r+(q+q)=r.
\]
Hence $F(f_z)=0$, whereas $f_z$ is nonzero and its image is central.
This argument uses only the displayed equation $q+q=0$; it does not assume
that arbitrary hom-monoids of $\A$ are groups.
\end{proof}

\begin{lemma}\label{bic:cor:finite-group-unit-not-faithful}
The unit $b_{\FinGrp}:\FinGrp\to\Bic(\FinGrp)$ is not faithful.
More precisely, it sends the nonzero homomorphism $f_z$ in the theorem
to zero.
\end{lemma}
\begin{proof}
The target is linear by Lemma~\ref{bic:prop:invertible-comparison},
and the unit preserves the specified product, which is the ordinary
direct product. For nonfaithfulness alone, the first part of the proof
above suffices: it gives $b_{\FinGrp}(f_t)=b_{\FinGrp}(f_{t'})$ for
the distinct homomorphisms $f_t,f_{t'}$. The stronger assertion about
$f_z$ follows from
Theorem~\ref{lin:thm:finite-group-no-faithful-completion}.
\end{proof}

\section*{Acknowledgement}
This paper is based on a chapter of the PhD thesis of the first author,
which grew out of a research discussion in 2023, within the Mathematical
Structures research programme of NITheCS involving the two authors,
Michael Hoefnagel, Charles Msipha, and Emma Theart, organised on the occasion of the visit of Charles Msipha. We are grateful for
their input, particularly to Charles Msipha, whose interest in monoidal
categories inspired the development of the present research topic. When finalising the paper, the second author used research and writing assistance from AI (ChatGPT-6 Astra), guided by him in a long routine of back and forth interaction with the model. Both authors remain responsible for the quality of the final paper as far as correctness of results and human originality in the development of these results and the ideas behind them are concerned and have spent a significant amount of time in bringing the paper to its current form.


\begin{thebibliography}{99}

\bibitem{AguiarMahajan2010}
M.~Aguiar and S.~Mahajan,
\emph{Monoidal Functors, Species and Hopf Algebras},
CRM Monograph Series, vol.~29,
American Mathematical Society, Providence, RI, 2010.

\bibitem{ArbibManes1980}
M.~A.~Arbib and E.~G.~Manes,
Partially additive categories and flow-diagram semantics,
\emph{J. Algebra} \textbf{62} (1980), 203--227.

\bibitem{BataninMarkl2012}
M.~Batanin and M.~Markl,
Centers and homotopy centers in enriched monoidal categories,
\emph{Adv. Math.} \textbf{230} (2012), no.~4--6, 1811--1858.

\bibitem{BezhanishviliEtAl2021}
G.~Bezhanishvili, N.~Bezhanishvili, L.~Carai, D.~Gabelaia,
S.~Ghilardi and M.~Jibladze,
Diego's theorem for nuclear implicative semilattices,
\emph{Indag. Math.} \textbf{32} (2021), no.~2, 498--535.

\bibitem{BonchiCioffo2026}
F.~Bonchi and C.~J.~Cioffo,
Tapes as stochastic matrices of string diagrams,
in \emph{Foundations of Software Science and Computation Structures
(FoSSaCS 2026)}, Lecture Notes in Computer Science, vol.~16503,
Springer, 2026, 88--109.

\bibitem{BookerStreet2013}
T.~Booker and R.~Street,
Tannaka duality and convolution for duoidal categories,
\emph{Theory Appl. Categ.} \textbf{28} (2013), no.~6, 166--205.

\bibitem{BorsikDobos1981}
J.~Bors\'ik and J.~Dobo\v{s},
On a product of metric spaces,
\emph{Math. Slovaca} \textbf{31} (1981), no.~2, 193--205.

\bibitem{Cho2015}
K.~Cho,
Total and partial computation in categorical quantum foundations,
\emph{Electron. Proc. Theor. Comput. Sci.} \textbf{195} (2015), 116--135.

\bibitem{CockettLack2007}
J.~R.~B.~Cockett and S.~Lack,
Restriction categories III: colimits, partial limits, and extensivity,
\emph{Math. Structures Comput. Sci.} \textbf{17} (2007), no.~4, 775--817.

\bibitem{CockettSeely2001}
J.~R.~B.~Cockett and R.~A.~G.~Seely,
Finite sum--product logic,
\emph{Theory Appl. Categ.} \textbf{8} (2001),
no.~5, 63--99.

\bibitem{DosenPetric2002}
K.~Do\v{s}en and Z.~Petri\'{c},
Bicartesian coherence,
\emph{Studia Logica} \textbf{71} (2002),
no.~3, 331--353.

\bibitem{DosenPetric2009}
K.~Do\v{s}en and Z.~Petri\'{c},
Bicartesian coherence revisited,
in Z.~Ognjanovi\'{c} (ed.),
\emph{Logic in Computer Science},
Zbornik radova, vol.~12(20),
Matemati\v{c}ki institut SANU, Belgrade, 2009, 5--34.

\bibitem{Ehrhard2023}
T.~Ehrhard,
Coherent differentiation,
\emph{Math. Structures Comput. Sci.} \textbf{33} (2023), no.~4--5,
259--310.

\bibitem{EhrhardWalch2025}
T.~Ehrhard and A.~Walch,
Coherent Taylor expansion as a bimonad,
\emph{Math. Structures Comput. Sci.} \textbf{35} (2025), e9.

\bibitem{FreyScedrov1990}
P.~J.~Freyd and A.~Scedrov,
\emph{Categories, Allegories},
North-Holland Mathematical Library, vol.~39,
North-Holland, Amsterdam, 1990.

\bibitem{heyting2025}
X.~Garc\'ia-Mart\'inez, J.~R.~A.~Gray, M.~A.~Hoefnagel,
T.~Van~der~Linden and C.~Vienne,
Categorical-algebraic aspects of Heyting semilattices,
\emph{Ann. Pure Appl. Logic} \textbf{178} (2027), no.~1, 103823.

\bibitem{Garner2009}
R.~Garner,
Understanding the small object argument,
\emph{Appl. Categ. Structures} \textbf{17} (2009), no.~3, 247--285.

\bibitem{Grothendieck1957}
A.~Grothendieck,
Sur quelques points d'alg\`ebre homologique,
\emph{T\^ohoku Math. J. (2)} \textbf{9} (1957), no.~2, 119--183.

\bibitem{Haghverdi2000}
E.~Haghverdi,
Unique decomposition categories, Geometry of Interaction and combinatory logic,
\emph{Math. Structures Comput. Sci.} \textbf{10} (2000), no.~2, 205--230.

\bibitem{HaghverdiScott2006}
E.~Haghverdi and P.~Scott,
A categorical model for the geometry of interaction,
\emph{Theoret. Comput. Sci.} \textbf{350} (2006), 252--274.

\bibitem{Ishiki2023}
Y.~Ishiki,
Metric trees in the Gromov--Hausdorff space,
\emph{Comment. Math. Univ. Carolin.} \textbf{64} (2023),
no.~1, 73--82.

\bibitem{Janelidze_2009}
Z.~Janelidze,
Cover relations on categories,
\emph{Appl. Categ. Structures} \textbf{17} (2009), no.~4, 351--371.

\bibitem{KaarsgaardAxelsenGluck2017}
R.~Kaarsgaard, H.~B.~Axelsen and R.~Gl\"uck,
Join inverse categories and reversible recursion,
\emph{J. Log. Algebr. Methods Program.} \textbf{87} (2017), 33--50.

\bibitem{Lack2004}
S.~Lack,
Composing PROPs,
\emph{Theory Appl. Categ.} \textbf{13} (2004), no.~9, 147--163.

\bibitem{Lack2010}
S.~Lack,
A 2-categories companion,
in J.~C.~Baez and J.~P.~May (eds.),
\emph{Towards Higher Categories},
The IMA Volumes in Mathematics and its Applications, vol.~152,
Springer, New York, 2010, 105--191.

\bibitem{Lawvere1992}
F.~W.~Lawvere,
Categories of space and of quantity,
in J.~Echeverr\'ia, A.~Ibarra and T.~Mormann (eds.),
\emph{The Space of Mathematics: Philosophical, Epistemological
and Historical Explorations},
de Gruyter, Berlin, 1992, 14--30.

\bibitem{Lawvere_Schanuel_2009}
F.~W.~Lawvere and S.~H.~Schanuel,
\emph{Conceptual Mathematics: A First Introduction to Categories},
second edition, Cambridge University Press, Cambridge, 2009.

\bibitem{Mitchell1965} B.~Mitchell,
\emph{Theory of Categories},
Pure and Applied Mathematics, vol.~17,
Academic Press, New York and London, 1965.

\bibitem{MacLane1950}
S.~Mac~Lane,
Duality for groups,
\emph{Bull. Amer. Math. Soc.} \textbf{56} (1950), no.~6, 485--516.

\bibitem{maclane_1998}
S.~Mac~Lane,
\emph{Categories for the Working Mathematician},
second edition, Graduate Texts in Mathematics, vol.~5,
Springer, New York, 1998.

\bibitem{MartinsFerreira2008Thesis}
N.~Martins-Ferreira,
\emph{Low-dimensional internal categorial structures in weakly
Mal'cev sesquicategories},
PhD thesis, University of Cape Town, 2008.

\bibitem{Mermin1990}
N.~D.~Mermin,
Simple unified form for the major no-hidden-variables theorems,
\emph{Phys. Rev. Lett.} \textbf{65} (1990), no.~27, 3373--3376.

\bibitem{ShapiroSpivak2022}
B.~T.~Shapiro and D.~I.~Spivak,
Duoidal structures for compositional dependence,
preprint (2022), revised 2025.

\bibitem{Street2012}
R.~Street,
Monoidal categories in, and linking, geometry and algebra,
\emph{Bull. Belg. Math. Soc. Simon Stevin} \textbf{19} (2012),
no.~5, 769--821.

\bibitem{TassonWalch2026}
C.~Tasson and A.~Walch,
Absolute convergence and Taylor expansion in web based models of linear logic,
in: \emph{11th International Conference on Formal Structures for
Computation and Deduction (FSCD 2026)},
\emph{Leibniz Int. Proc. Inform.} \textbf{378} (2026), 29:1--29:23.

\bibitem{Tull2018}
S.~Tull,
\emph{Categorical operational physics},
DPhil thesis, University of Oxford, 2018.

\end{thebibliography}
\end{document}